\documentclass{article}
\usepackage{graphicx} % Required for inserting images
\usepackage{amssymb}
\usepackage{authblk}
\usepackage{amscd}
\usepackage{caption}
\usepackage{blkarray}
\usepackage{mathrsfs}
\usepackage{xcolor}
\usepackage{bm}
\usepackage[centertags]{amsmath}
\usepackage{amssymb}
\usepackage[english]{babel}
\usepackage{blindtext}
\usepackage{tikz}
\usepackage{amsthm}
\usepackage{graphicx}
\usepackage{booktabs}   % For \toprule, \midrule, \bottomrule
\usepackage{multirow}   % If multi-row cells are needed
\usepackage{siunitx}    % For easy alignment of numbers

\usepackage{gnuplottex}
\usepackage{epsfig}
\usepackage[belowskip=-15pt,aboveskip=0pt]{caption}
\usepackage{amsthm}
\usepackage{amssymb,latexsym}
\usepackage{amsfonts,amsmath}
\usepackage[T1]{fontenc}
\usepackage{amsmath,amssymb}
\usepackage{latexsym}
\numberwithin{equation}{section}
\usepackage{amssymb}
\usepackage{amsfonts}
\usepackage[arrow,frame,matrix]{xy}

\usepackage{fullpage}
\usepackage{hyperref}
\usepackage[margin=2.6cm]{geometry}
\usepackage{placeins}
\newtheorem{definition}{Definition}
\usepackage{grffile}
\usetikzlibrary{decorations.pathreplacing}
\usepackage{color}
\usepackage{float}
\usepackage{subcaption}
\usepackage{array}   % for extended column formatting
\usepackage{cancel}
\usepackage{ragged2e}
\usepackage[strings]{underscore}
\usepackage{parskip}
\usepackage{float}
\usepackage{placeins}
\usepackage{enumitem} % for itemize customization
\newtheorem{theorem}{Theorem}
\newtheorem{remark}{Remark}
\newtheorem{conjecture}{Conjecture}

\title{Geometric means of HPD GLT matrix-sequences: a maximal result beyond invertibility assumptions on the GLT symbols}
\author[1]{Asim Ilyas}
\author[1]{Muhammad Faisal Khan}
\author[1]{Valerio Loi}
\author[1,2]{S. Serra-Capizzano}
\affil[1]{Department of Science and High Technology, University of Insubria, Italy; email: ailyas1@uninsubria.it; mfkhan@uninsubria.it; vloi@uninsubria.it s.serracapizzano@uninsubria.it; }
\affil[2]{Department of Information Technology, University of Uppsala, Sweden; email: stefano.serra@it.uu.se}
\date{} % This removes the date

\begin{document}

\maketitle

\begin{abstract}
    In the current work, we consider the study of the spectral distribution of the geometric mean matrix-sequence of two matrix-sequences 
$\{G(A_n, B_n)\}_n$   formed by Hermitian Positive Definite (HPD) matrices,  
     assuming that the two input matrix-sequences $\{A_n\}_n, \{B_n\}_n$ belong to the same $d$-level $r$-block Generalized Locally Toeplitz (GLT) $\ast$-algebra with $d,r\ge 1$ and with GLT symbols $\kappa, \xi$. Building on recent results in the literature,
we examine whether the assumption that at least one of the input GLT symbols is invertible almost everywhere (a.e.) is necessary. Since inversion is mainly required due to the non-commutativity of the matrix product, it was conjectured that the hypothesis on the invertibility of the GLT symbols can be removed.
In fact, we prove the conjectured statement that is 
\[
\{G(A_n, B_n)\}_n \sim_{\mathrm{GLT}} (\kappa \xi)^{1/2}
\]
when the symbols $\kappa, \xi$ commute, which implies the important case where $r=1$ and $d \geq 1 $, while the statement is generally false or even not well posed when the symbols are not invertible a.e. and do not commute. 
In fact, numerical experiments are conducted in the case where the two symbols do not commute, showing that the main results of the present work are maximal. Further numerical experiments, visualizations, and conclusions end the present contribution.

%The study also extends the proof for the multilevel block case (\( r=1 \) and \( d \geq 1 \)) while refining the numerical validation in broader settings.
\end{abstract}

\section{Introduction}\label{intro}

The concept of the matrix geometric mean has attracted considerable attention from scholars in recent decades due to its elegant theoretical foundations and growing significance across numerous mathematical, engineering, and applied sciences; see e.g. \cite{batchelor2005rigorous,yang2010geometry,lapuyade2008radar,yger2017riemannian,Rathi2007,moakher2006averaging,fasi2018computing} and references therein. Initially introduced implicitly in the context of a functional calculus for sesquilinear maps by Pusz and Woronowicz \cite{pusz1975functional}, but popularized as a mean by Kubo and Ando \cite{kubo1980means}, the matrix geometric mean was rigorously formalized as follows: for two positive definite matrices $A$ and $B$ the geometric mean $G(A,B)$ is defined as $G(A,B) = A^{\frac{1}{2}} (A^{-\frac{1}{2}} B A^{-\frac{1}{2}})^{\frac{1}{2}} A^{\frac{1}{2}}=G(B,A)$ The rigorous definition has been analyzed by Ando, Li, and Mathias (ALM) \cite{ando2004geometric}, who also identified essential axiomatic properties \cite[ Section 3]{Bini2024} that a proper matrix geometric mean should satisfy. These axioms serve as a foundation for several established definitions, including the ALM mean itself, the Nakamura-Bini-Meini-Poloni (NBMP) mean \cite{Nakamura2009}, and the Karcher mean, recognized as a matrix geometric mean through Riemannian geometry as introduced in \cite{moakher2005differential,bini2013computing}. Further detailed historical developments and generalizations can be found in \cite{Bhatia2007}.\\
Matrix geometric means have found wide-ranging applications including Diffusion Tensor Imaging (DTI) \cite{batchelor2005rigorous}, radar detection \cite{yang2010geometry,lapuyade2008radar}, image processing \cite{Rathi2007}, elasticity \cite{moakher2006averaging}, machine learning \cite{iannazzo2019derivative}, brain-computer interfaces \cite{yger2017riemannian}, and network analysis \cite{fasi2018computing}. Its relevance becomes particularly significant in applications involving structured matrix-sequences arising from discretizations of differential operators, where hidden asymptotic structures are often revealed through spectral analysis. Such structured sequences frequently belong to the class of Locally Toeplitz (LT) sequences, thereby closely connecting the study of geometric means with the asymptotic spectral properties described by the theory of LT sequences.\\
The theory of LT sequences provides a powerful framework for analyzing the asymptotic singular value and eigenvalue distributions of structured matrix-sequences that arise in numerical analysis. In the context of partial differential equation (PDE) discretization, the matrices resulting from numerical approximations often fail to retain the classical Toeplitz structure due to the presence of variable coefficients. Traditionally, Toeplitz matrices have played a central role in problems involving translation-invariant differential operators, where constant coefficients ensure a globally repeating structure. Instead, the resulting matrices exhibit a Toeplitz-like structure that varies smoothly along the diagonals as the matrix size increases, leading to the broader class of LT sequences \cite{tilli1998locally, tyrtyshnikov1996unifying}.\\
A classical Toeplitz matrix $T_n$ of size $n \times n$ is generated by a function $f \in L^1([-\pi, \pi]^d)$, where the constant diagonal entries correspond to the Fourier coefficients of $f$. However, for differential operators with non-constant coefficients, the resulting matrices do not maintain this global Toeplitz structure but instead belong to the LT class. Specifically, an LT matrix $A_{\bm{n}}$ is obtained by modulating the entries of a Toeplitz matrix using an almost everywhere (a.e.) continuous function $a(x): [0,1]^d \to \mathbb{C}$, scaling the entries along the same diagonal
$ a\left(\frac{i}{n}\right)$ for $i= 1,\dots,n.$ As $n \to \infty$, the differences between consecutive diagonal elements vanish due to continuity, making the structure asymptotically Toeplitz-like. More precisely, this structure can be represented as the product of a Toeplitz matrix and a diagonal matrix, where the diagonal entries correspond to evaluations of the function $a(x)$, such that $A_{n} = D_{n}(a)T_{n},$
where $D_{n}(a)$ is a diagonal matrix whose entries are given by $D_{n}(a) = \mathrm{diag}\left(a\left(\frac{i}{n}\right)\right)_{i=1}^{n}.$
This formulation allows LT sequences to accurately capture local variations in variable-coefficient differential operators while preserving the structural properties of Toeplitz matrices. Formally, an LT sequence $\{A_{n}\}_{n}$ can be expressed as the tensor product of the weight function $a(x)$ and the generating function $f(\theta)$, leading to the LT symbol $
a \otimes f: [0,1]^d \times [-\pi, \pi]^d \to \mathbb{C}$.\\
Building on this concept, a Generalized Locally Toeplitz (GLT) matrix sequence is defined as the limit, in the sense of the approximating class of sequences (a.c.s.), as described in Definition \ref{sec:acs}), of a finite sum of LT sequences. The construction of GLT sequences was formalized through a.c.s. limits in \cite{serra2003generalized} and refined to a \textbf{d-level $*$-algebra of structured matrix sequences} satisfying specific axioms, listed in \cite[Chapter~6, pp.~118--120]{garoni2018}. The theory of GLT sequences provides a fundamental framework for analyzing structured matrix-sequences arising from discretizations of differential equations. To each GLT matrix-sequence $\{A_n\}_n$, there corresponds a measurable function $\kappa:[0,1]^d \times [-\pi,\pi]^d \to \mathbb{C}$, known as the GLT symbol. This symbol characterizes the asymptotic spectral behavior of the sequence and is formally expressed as $\{A_n\}_n \sim_{\mathrm{GLT}} \kappa.$ A fundamental property of the GLT symbol is its uniqueness, stated formally as follows: If $\{A_n\}_n \sim_{\mathrm{GLT}} \kappa$ and $\{B_n\}_n \sim_{\mathrm{GLT}} \xi$, then $
\kappa=\xi,$ \text{almost every (a.e.)}. The GLT symbol fully characterizes the asymptotic spectral distribution of singular values $\{A_n\}_n \sim_{\sigma} \kappa$. Furthermore, if the matrices are Hermitian, the symbol is almost everywhere real-valued and describes the eigenvalue distribution $\{A_n\}_n \sim_{\lambda} \kappa$ precisely. \\
Building upon the foundational concept of GLT symbols, a recent refinement known as Toeplitz/GLT momentary symbols has been introduced in ~\cite{bolten2022toeplitz, bolten2023note} to further enhance the spectral approximation of structured matrices. While classical GLT theory primarily focuses on the asymptotic regime and disregards small norm and low-rank perturbations, momentary symbols explicitly retain small norm contributions, leading to a more accurate characterization of spectral distributions even for moderate matrix dimensions. Formally, the momentary symbol addresses structured matrix-sequences of the form $
\{X_n\}_n = \{T_n(f)\}_n + \{N_n\}_n + \{R_n\}_n,
$ where $T_n(f)$ is a Toeplitz matrix generated by a function $f$, $\{N_n\}_n$ represents a small-norm perturbation sequence, and $\{R_n\}_n$ denotes a low-rank matrix sequence. In traditional GLT theory, an admissible small norm sequence $\{N_n\}_n$ must take the specific form $g(n)T_n(f)$, where $g(n) \to 0$ as $n \to \infty$. In contrast, the framework of momentary symbols admits more general choices for the perturbation $g(n)$, not necessarily tied to such strict asymptotic constraints. Although low-rank contributions $\{R_n\}_n$ are still disregarded in constructing momentary symbols due to their negligible influence on the asymptotic spectral distribution, retaining small norm perturbations significantly enhances spectral approximations for finite-size matrices. \\
An initial spectral study on the geometric mean of structured matrix-sequences was carried out in \cite{ahmad2025matrix}, where the spectral properties of HPD GLT matrix-sequences were analyzed in the context of geometric means. The study first established that for two HPD $d$-level GLT matrix-sequences $\{A_n\}_n \sim_{\mathrm{GLT}}  \kappa$ and $\{B_n\}_n \sim_{\mathrm{GLT}} \xi$, their geometric mean also forms a GLT sequence with the symbol
$
\{G(A_n, B_n)\}_n \sim_{\mathrm{GLT}} (\kappa \xi)^{1/2},
$
as rigorously proven in \cite[Theorem 3, Theorem 4]{ahmad2025matrix}, under the assumption that either $\kappa$ or $\xi$ are nonzero almost everywhere. This result was  generalized to $d$-level $r$-block GLT sequences in \cite[Theorem 5]{ahmad2025matrix}, extending its applicability to the case where the symbols $\kappa$ and $\xi$ do not necessarily commute. In that case we have again 
$
\{G(A_n, B_n)\}_n \sim_{\mathrm{GLT}} G(\kappa, \xi),
$
again under the assumption either $\kappa$ or $\xi$ are invertible almost everywhere.

\subsection*{Main results}

Building on our previous results \cite{ahmad2025matrix}, this work further investigates the necessity of the assumption that at least one of the input GLT symbols is invertible almost everywhere. This assumption was initially introduced to ensure that the inverse of a GLT matrix-sequence remains within the GLT framework. However, since inversion is primarily required due to the non-commutativity of matrices, we explore whether this condition can be relaxed in cases where the symbols commute, in the context of $d$-level $r$-block GLT matrix-sequences, $d, r \geq 1$. 
%To support this idea, we conduct extensive numerical experiments, analyzing whether the spectral properties of the geometric mean remain preserved when the assumption is removed. 
We prove formally that the invertibility assumption on the GLT symbols can be removed when the GLT symbols commute, which includes the case where $r=1$ and $ dge 1$. In the general setting, we perform extensive numerical experiments. The numerical tests suggest that the requirement of invertibility is essential, showing that our results are maximal when $r>1$.
Indeed, for $r>1$ we report and discuss several cases in which $G(A_n, B_n)$ is well defined for any $n$, $\{A_n\}_n \sim_{\mathrm{GLT}} \kappa$,
$\{B_n\}_n \sim_{\mathrm{GLT}} \xi$, the GLT symbols $\kappa,\xi$ do not commute, and are both zero in a set of positive measure. The resulting observation is that
\[
\{G(A_n, B_n)\}_n \sim_{\mathrm{GLT}} \psi,
\]
but either $\psi$ does not coincide with $G(\kappa, \xi)$ or even $G(\kappa, \xi)$ is not well defined.

Further numerical experiments are also conducted to investigate finer spectral aspects. More precisely, we study the relationships between the order to zeros of the geomeric mean of the symbols minus its minimum and the convergence speed to this minimul of the minimal eigenvalue of the geometric meand of the corresponding GLT matrix-sequences, mimicking what is known in Toeplitz \cite{SerraCapizzano1996,SerraCapizzano1998,Bottcher1998,SerraCapizzano1999a,SerraCapizzano1999b,SerraCapizzanoTilli1999} and GLT setting \cite{extr2-glt,extr2-glt}, so opening new directions for generalizing the existing GLT theory. Notably, in additional numerical tests e.g. when the product $\kappa \xi$ is equal to zero, the eigenvalue distribution of the geometric mean obtained numerically seems to possess a richer structure. More specifically, the experiments show that the resulting asymptotic spectra can be described by exploiting the theory of GLT momentary symbols, according to \cite{bolten2022toeplitz,bolten2023note,new momentary}.\\

\subsubsection*{Structure of the work}

The present study is structured as follows. In Section \ref{sec:spectral}, we introduce notations, terminology, and preliminary results related to Toeplitz and GLT structures, essential for the mathematical formulation and technical solution of the problem. In Section \ref{sec GM}, we present the main results where we drop the assumption that at least one GLT symbol is invertible when the symbols commute, which includes the case of $d\geq 1,r=1$. 
In Section \ref{Num_Exp}, we present numerical experiments illustrating the asymptotic spectral behavior of the geometric mean for GLT matrix-sequences in both 1D and 2D cases, considering both scalar and block structures. 
In particular we give evidence of the extremal spectral behavior, of the emergence of GLT momentary symbols phenomena, and more importantly we give evidence that the results of Section \ref{sec GM} are maximal, since the case of noncommuting GLT symbols ($r>1$) which are both noninvertible leads to noncanonical distribution results. 
Finally, in Section \ref{sec concl}, we draw conclusions and highlight several problems for future research.

\section{Spectral tools}\label{sec:spectral}

In this section, we introduce the essential tools for the spectral analysis of the matrices under consideration, using the multi-level block GLT matrix sequence framework. The case of scalar values, corresponding to the non-block setting, has been extensively detailed in \cite{garoni2017,garoni2018}, while the block setting, associated with matrix-valued symbols, is thoroughly examined in \cite{Barbarino2020a,Barbarino2020b}. In our specific context, where the block size is $r=1$ and the problem domain has a dimensionality of $d=2$, we operate within the realm of two-level non-block GLT sequences.
\subsection{Notation and terminology}
\textbf{Matrices and matrix-sequences}. Given a square matrix $A \in \mathbb{C}^{m \times m}$, we denote by $A^*$ its conjugate transpose and by $A^\dagger$ the Moore–Penrose pseudoinverse of $A$. Recall that $A^\dagger = A^{-1}$ whenever $A$ is invertible. The singular values and eigenvalues of $A$ are denoted respectively by $ \sigma_1(A), \cdots, \sigma_m(A) $ and $\lambda_1(A), \cdots, \lambda_m(A) $.\\
Regarding matrix norms, $\|\cdot\|$ refers to the spectral norm, and for $1 \leq p \leq \infty$, the notation $\|\cdot\|_p$ stands for the Schatten $p$-norm defined as the $p$-norm of the vector of singular values. Note that the Schatten $\infty$-norm, which is equal to the largest singular value, coincides with the spectral norm $\|\cdot\|$; the Schatten $1$-norm since it is the sum of the singular values is often referred to as the trace-norm; and the Schatten $2$-norm coincides with the Frobenius norm. Schatten $p$-norms, as important special cases of unitarily invariant norms, are treated in detail in a wonderful book by Bhatia \cite{Bhatia1997}.\\
Finally, the expression \textbf{matrix-sequence} refers to any sequence of the form $\{A_n\}_n$, where $A_n$ is a square matrix of size $d_n$ with $d_n$ strictly increasing so that $d_n \to \infty$ as $n \to \infty$. A \textbf{r-block matrix-sequence}, or simply a matrix-sequence if $r$ can be deduced from context is a special $\{A_n\}_n$ in which the size of $A_n$ is $d_n = r\varphi_n$, with $r \geq 1 \in \mathbb{N}$ fixed and $\varphi_n \in \mathbb{N}$ strictly increasing.
\subsection{Multi-index notation}
To effectively deal with multilevel structures, it is necessary to use multi-indices, which are vectors of the form $\textbf{i} = (i_1, \cdots, i_d) \in \mathbb{Z}^d$. The related notation is listed below
\begin{itemize}
    \item $\bm{0}, \bm{1}, \bm{2}, \dots$ are vectors of all zeroes, ones, twos, etc.
    \item $\textbf{h} \leq \textbf{k}$ means that $h_r \leq k_r$ for all $r = 1, \cdots, d$. In general, relations between multi-indices are evaluated componentwise.
    \item Operations between multi-indices, such as addition, subtraction, multiplication, and division, are also performed componentwise.
    \item The multi-index interval $[\textbf{h}, \textbf{k}]$ is the set $\{\textbf{j} \in \mathbb{Z}^d : \textbf{h} \leq \textbf{j} \leq \textbf{k}\}$. We always assume that the elements in an interval $[\textbf{h}, \textbf{k}]$ are ordered in the standard lexicographic manner
  \[
\Biggr[ \cdots
\biggr[\Bigr[\bigl(j_1, \cdots, j_d\bigl)\Bigr]
_{j_d = h_d, \cdots, k_d}\biggr]
_{j_{d-1} = h_{d-1}, \cdots, k_{d-1}}
\cdots \Biggr]_{j_1 = h_1, \cdots, k_1}.
\]
    \item $\textbf{j} = \textbf{h}, \cdots, \textbf{k}$ means that $\textbf{j}$ varies from $\textbf{h}$ to $\textbf{k}$, always following the lexicographic ordering.
    \item $\textbf{m} \to \infty$ means that $\min(\textbf{m}) = \min_{j=1, \cdots, d} m_j \to \infty$.
    \item The product of all the components of $\textbf{m}$ is denoted as $\nu(\textbf{m}) := \prod_{j=1}^{d} m_j$.
\end{itemize}
A \textbf{multilevel matrix-sequence} is a matrix-sequence $\{A_{\textbf{n}} \}_n$ such that $n$ varies in some infinite subset of $\mathbb{N}$, $\textbf{n} = \textbf{n}(n)$ is a multi-index in $\mathbb{N}^d$ depending on $n$, and $\textbf{n} \to \infty$ when $n \to \infty$. This is typical of many approximations of differential operators in $d$ dimensions.\\
\textbf{Measurability.} All the terminology from measure theory, such as “measurable set”, “measurable function”, “a.e.”, etc., refers to the Lebesgue measure in $ \mathbb{R}^t $, denoted with $\mu_t $. A matrix-valued function $f : D \subseteq \mathbb{R}^t \to \mathbb{C}^{r \times r} $ is said to be measurable (resp., continuous, 
Riemann-integrable, in $ L^p(D) $, etc.) if all its components $ f_{\alpha\beta} : D \to \mathbb{C} $,  $ \alpha, \beta = 1, \dots, r $, are measurable (resp., continuous, Riemann-integrable, in $ L^p(D) $, etc.). If $ f_m, f : D \subseteq \mathbb{R}^t \to \mathbb{C}^{r \times r} $ are measurable, we say that $ f_m $ converges to $ f $ in measure (resp., a.e., in $ L^p(D) $, etc.) if $ (f_m)_{\alpha\beta} $ converges to 
$ f_{\alpha\beta} $ in measure (resp., a.e., in $ L^p(D) $, etc.) for all $ \alpha, \beta = 1, \cdots, r $. If $A \in \mathbb{C}^{n \times n}$, the singular values and eigenvalues of $A$ are denoted by $\sigma_1(A), \dots, \sigma_n(A)$ and $\lambda_1(A), \dots, \lambda_n(A)$, respectively. The set of eigenvalues (i.e., the spectrum) of $A$ is denoted by $\Lambda(A)$.
\subsection{ Singular Value and Eigenvalue Distributions of a Matrix-Sequence}
\begin{definition}\label{def:sing and eig}
Let $\{A_n\}_n$ be a matrix sequence, where $A_n$ is of size $d_n$, and let $\psi: D \subset \mathbb{R}^t \to \mathbb{C}^{r \times r}$ be a measurable function defined on a set $D$ with $0 < \mu_t(D) < \infty$.
\begin{itemize}
    \item We say that $\{A_n\}_n$ has an \textbf{(asymptotic) singular value distribution} described by $\psi$, written as $\{A_n\}_n \sim_{\sigma} \psi$, if
    \begin{equation*}
        \lim_{n \to \infty} \frac{1}{d_n} \sum_{i=1}^{d_n} F(\sigma_i(A_n)) = \frac{1}{\mu_t(D)} \int_D \displaystyle\frac{\sum_{i=1}^{r} F(\sigma_i(\psi(\textbf{x})))}{r} d\textbf{x}, \quad \forall F \in C_c(\mathbb{R}).
    \end{equation*}
    \item We say that $\{A_n\}_n$ has an \textbf{(asymptotic) spectral (eigenvalue) distribution} described by $\psi$, written as $\{A_n\}_n \sim_{\lambda} \psi$, if
    \begin{equation*}
        \lim_{n \to \infty} \frac{1}{d_n} \sum_{i=1}^{d_n} F(\lambda_i(A_n)) = \frac{1}{\mu_t(D)} \int_D \displaystyle\frac{\sum_{i=1}^{r} F(\lambda_i(\psi(\textbf{x})))}{r} d\textbf{x}, \quad \forall F \in C_c(\mathbb{C}).
    \end{equation*}
\end{itemize}
If $\psi$ describes both singular value and eigenvalue distribution of $\{A_n\}_n$, we write $\{A_n\}_n \sim_{\sigma, \lambda} \psi$.\\
 In this case, the function $\psi$ is referred to as the \textit{eigenvalue (or spectral) symbol} of $\{A_n\}_n$. The same definition when the considered matrix-sequence shows a multilevel structure. In that case, $n$ is replaced by $\bm{n}$ uniformly in $A_n$ and $d_n$.
\end{definition}
The informal meaning behind the spectral distribution definition is the following: if $\psi$
 is continuous, then a suitable ordering of the eigenvalues $\{\lambda_j(A_n)\}_{j=1,\cdots,d_n}$,
 assigned in correspondence with an equispaced grid on $D$, reconstructs approximately the $r$ surfaces $\textbf{x} \to \lambda_i(\psi(\textbf{x})),  i =1,...,r.$ For instance, in the simplest case where $t =1$ and $D =[a,b]$, 
$d_n = nr$, the eigenvalues of $A_n$ are approximately equal - up to few potential outliers to  $\lambda_i (\psi(x_j)),$ where
 \begin{equation*}
  x_j = a + j \frac{(b-a)}{n}, \quad j =1,\cdots,n, \quad i =1,\cdots,r.
 \end{equation*}
If $t =2$ and $D = [a_1,b_1] \times [a_2,b_2]$, 
$d_n = n^2r$, the eigenvalues of $A_n$ are approximately equal, again up to a few potential outliers, to $\lambda_i (\psi(x_{j_1},y_{j_2})), i =1,\cdots,r,$ where
 \begin{align*}
 x_{j_1} &= a_1 + j_1 \frac{(b_1 - a_1)}{n}, \qquad j_1 =1,\cdots,n,\\
 y_{j_2} &= a_2 + j_2 \frac{(b_2 - a_2)}{n},\qquad j_2 =1,\cdots,n,
 \end{align*} 
If the considered structure is two-level, then the subscript is $\bm{n}=(n_1,n_2)$ and $d_n = n_1 n_2 r$. Furthermore, for $t \geq 3$, a similar reasoning applies.\\
Finally, we report an observation that is useful in the following derivations.
\begin{remark}\label{rem: range}
The relation $\{A_n\}_n \sim_\lambda f$ and $\Lambda(A_n) \subseteq S$ for all $n$ imply that the range of $f$ is a subset of the closure $\bar{S}$ of 
$S$. In particular, $\{A_n\}_n \sim_\lambda f$ and $A_n$ positive definite for all $n$ imply that $f$ is non-negative definite almost everywhere, simply nonnegative almost everywhere if $r=1$. The same applies when a multilevel matrix sequence $\{A_{\bm{n}}\}_{\bm{n}}$ is considered and similar statements hold when singular values are taken into account.
\end{remark}
\subsection{Approximating Classes of Sequences}\label{sec:acs}
In this subsection, we present the notion of the approximating class of sequences and a related key result.
\begin{definition}\label{def:acs}
\textbf{(Approximating class of sequences)}  
Let $\{A_n\}_n$ be a matrix-sequence and let $\{\{B_{n,j}\}_n\}_j$ be a class of matrix-sequences, with $A_n$ and $B_{n,j}$ of size $d_n$.  
We say that $\{\{B_{n,j}\}_n\}_j$ is an approximating class of sequences (a.c.s.) for $\{A_n\}_n$ if the following condition is met:  
for every $j$, there exists $n_j$ such that, for every $n \geq n_j$,
\begin{equation*}
    A_n = B_{n,j} + R_{n,j} + N_{n,j},
\end{equation*}
where
\begin{equation*}
    \text{rank } R_{n,j} \leq c(j) d_n, \quad \text{and} \quad \|N_{n,j}\| \leq \omega(j),
\end{equation*}
where $n_j$, $c(j)$, and $\omega(j)$ depend only on $j$ and
\begin{equation*}
    \lim_{j \to \infty} c(j) = \lim_{j \to \infty} \omega(j) = 0.
\end{equation*}
$
\{ \{ B_{n,j} \}_n \}_j \xrightarrow{\text{a.c.s. wrt } j} \{ A_n \}_n
$ denotes that $ \{ \{ B_{n,j} \}_n \}_j $ is an a.c.s. for $ \{ A_n \}_n $.
\end{definition}
The following theorem represents the expression of a related convergence theory and it is a powerful tool used, for example, in the construction of the GLT $*$-algebra.
\begin{theorem}\label{th:fundamental acs}
Let $\{A_n\}_n, \{B_{n,j}\}_n$, with $j,n \in \mathbb{N}$, be matrix-sequences and let $\psi, \psi_j : D \subset \mathbb{R}^d \to \mathbb{C}$ be measurable functions defined on a set $D$ with positive and finite Lebesgue measure. Suppose that
\begin{enumerate}
    \item $\{B_{n,j}\}_n \sim_{\sigma} \psi_j$ for every $j$;
    \item $\{\{B_{n,j}\}_n\}_j \xrightarrow{\text{a.c.s. wrt } j} \{A_n\}_n$;
    \item $\psi_j \to \psi$ in measure.
\end{enumerate}
Then
\begin{equation*}
    \{A_n\}_n \sim_{\sigma} \psi.
\end{equation*}
Moreover, if all the involved matrices are Hermitian, the first assumption is replaced by $ \{B_{n,j}\}_n \sim_{\lambda} \psi_j \; \text{for every } j,$ and the other two are left unchanged, then $\{A_n\}_n \sim_{\lambda} \psi.$
\end{theorem}
We end this section by observing that the same definition can be given and corresponding results (with obvious changes) hold, when the involved matrix-sequences show a multilevel structure. In that case $n$ is replaced by $\bm{n}$ uniformly in $A_{n}, B_{n,j}, d_{n}$.
\subsection{Matrix-Sequences with Explicit or Hidden (Asymptotic) Structure} \label{subsec:matrix_structures}
In this subsection, we introduce three types of matrix structures that serve as the fundamental building blocks of the GLT \(*\)-algebras. Specifically, for any positive integers \( d \) and \( r \), we consider the set of \( d \)-level \( r \)-block GLT matrix-sequences. This set forms a \(*\)-algebra of matrix-sequences, which is both maximal and isometrically equivalent to the maximal \(*\)-algebra of \( 2d \)-variate \( r \times r \) matrix-valued measurable functions (with respect to the Lebesgue measure) that are naturally defined over $ [0,1]^d \times [-\pi,\pi]^d; $ see \cite{Barbarino2020a, Barbarino2020b, garoni2017, garoni2018} and references therein.\\
The reduced version of this structure plays a crucial role in approximating integro-differential operators, including their fractional versions, particularly when defined over general (non-Cartesian) domains. This concept was initially introduced in \cite{serra2003generalized,SerraCapizzano2006} and later extensively developed in \cite{Barbarino2022}, where GLT symbols are again defined as measurable functions over $ \Omega \times [-\pi,\pi]^d,$ with \( \Omega \) being Peano-Jordan measurable and contained within \( [0,1]^d \). Additionally, the reduced versions also form maximal \(*\)-algebras that are isometrically equivalent to their corresponding maximal \(*\)-algebras of measurable functions.\\
These GLT \(*\)-algebras provide a rich framework of hidden (asymptotic) structures, built upon two fundamental classes of explicit algebraic structures: \( d \)-level \( r \)-block Toeplitz matrix-sequences and sampling diagonal matrix-sequences (discussed in Sections \ref{sec:multi} and \ref{sec:Block}), along with asymptotic structures described by zero-distributed matrix-sequences (see Section \ref{sec:zero}). Notably, the latter class serves an analogous role to compact operators in relation to bounded linear operators, forming a two-sided ideal of matrix-sequences within any of the GLT \(*\)-algebras.
\subsection{Zero-distributed sequences}\label{sec:zero}
 Zero-distributed sequences are defined as matrix sequences $\{A_n\}_n$ such that $\{A_n\}_n \sim_\sigma 0$. Note that, for any $r \geq 1$, $\{A_n\}_n \sim_\sigma 0$ is equivalent to 
$\{A_n\}_n \sim_\sigma O_r$, where $O_r$ is the $r \times r$ zero matrix. The following theorem, taken from \cite{SerraCapizzano2001, garoni2017}, provides a useful characterization for detecting this type of sequence. 
\begin{theorem}
Let $\{A_n\}_n$ be a matrix-sequence, with $A_n$ of size $d_n$ and let $p\in [1,\infty]$, with $\|X\|_p$ being the Schatten $p$-norm of $X$, that is the $l^p$ norm of the vector its singular values. Let $\|\cdot\|=\|\cdot\|_\infty$ be the spectral norm. With  
the natural convention $1/\infty = 0$, we have
\begin{itemize}
    \item $\{A_n\}_n \sim_\sigma 0$ if and only if $A_n = R_n + N_n$ with ${\operatorname{rank}(R_n)}/{d_n} \to 0$ and $ ||N_n|| \to 0$ as $n \to \infty$;
    \item $\{A_n\}_n \sim_\sigma 0$ if there exists $p \in [1,\infty]$ such that 
    \begin{equation*}
        \frac{\|A_n\|_p}{(d_n)^{1/p}} \to 0 \quad \text{as} \quad n \to \infty.
    \end{equation*}
\end{itemize}
\end{theorem}
As in Section \ref{sec:acs}, the same definition can be given and corresponding result (with obvious changes) holds, when the involved matrix-sequences show a multilevel structure.
In that case $n$ is replaced by $\bm{n}$ uniformly in $A_{n}, N_{n}, R_{n}, d_{n}$.

Further results connecting Schatten $p$-norms, unitarily invariant norms (see \cite{Bhatia1997}), variational characterization, and Toeplitz/Hankel structures can be found in \cite{SerraCapizzanoTilli-LPO}.

\subsection{ Multilevel block Toeplitz matrices}\label{sec:multi}
 Given \(\textbf{n} \in \mathbb{N}^d \), a matrix of the form
\begin{equation*}
[A_{\textbf{i}-\textbf{j}}]_{\textbf{i,j=1}}^{\textbf{n}} \in \mathbb{C}^{\nu(\textbf{n})r \times \nu(\textbf{n})r},
\end{equation*}
with blocks \( A_\textbf{k} \in \mathbb{C}^{r \times r} \), \( \textbf{k} \in \{-(\textbf{n}-1), \dots, \textbf{n}-1\} \), is called a multilevel block Toeplitz matrix, or, more precisely, a \( d \)-level \( r \)-block Toeplitz matrix.\\
Given a matrix-valued function \( f : [-\pi,\pi]^d \to \mathbb{C}^{r \times r} \) belonging to \( L^1([- \pi, \pi]^d) \), the \( \mathbf{n} \)-th Toeplitz matrix associated with \( f \) is defined as
\begin{equation*}
T_\mathbf{n}(f):=[\hat{f}_{\mathbf{i-j}}]_{\mathbf{i,j=1}}^{\mathbf{n}} \in \mathbb{C}^{\nu(\mathbf{n})r \times \nu(\mathbf{n})r},
\end{equation*}
where
\begin{equation*}
\hat{f}_\mathbf{k} = \frac{1}{(2\pi)^d} \int_{[-\pi,\pi]^d} f(\bm{\theta}) e^{-\hat{\iota} (\mathbf{k}, \bm{\theta)}} d\bm{\theta} \in \mathbb{C}^{r \times r}, \quad \mathbf{k} \in \mathbb{Z}^d,
\end{equation*}
are the Fourier coefficients of \( f \), in which \( \hat{\iota} \) denotes the imaginary unit, the integrals are computed componentwise, and \( (\mathbf{k}, \bm{\theta}) = k_1\theta_1 + \cdots + k_d\theta_d \). Equivalently, \( T_\mathbf{n}(f) \) can be expressed as
\begin{equation*}
T_\mathbf{n}(f) =
\sum_{|j_1|<n_1} \cdots \sum_{|j_d|<n_d} [J_{n_1}^{(j_1)}
\otimes \cdots \otimes J^{(j_d)}_{n_d}] \otimes \hat{f}(j_1, \cdots, j_d),
\end{equation*}
where \( \otimes \) denotes the Kronecker tensor product between matrices, and \( J^{(l)}_{m} \) is the matrix of order \( m \) whose \( (i,j) \) entry equals 1 if \( i - j = l \) and zero otherwise.\\
The family \( \{T_{\mathbf{n}}(f)\}_{\mathbf{n}\in\mathbb{N}^d} \) is the family of (multilevel block) Toeplitz matrices associated with \( f \), which is called the generating function.
\subsection{Block diagonal sampling matrices}\label{sec:Block}
 Given $d \geq 1$, $\mathbf{n} \in \mathbb{N}^d$ and a function $a : [0,1]^d \to \mathbb{C}^{r \times r}$, we define the multilevel block diagonal sampling matrix $D_{\mathbf{n}}(a)$ as the block diagonal matrix
\begin{equation*}
    D_{\mathbf{n}}(a) = \underset{\mathbf{i=1,\dots,n}}{\operatorname{diag}} a \left( \mathbf{\frac{i}{n}} \right) \in \mathbb{C}^{\nu(\mathbf{n})r \times \nu(\mathbf{n})r}.
\end{equation*}
\subsection{The $*$-Algebra of $d$-Level $r$-block GLT Matrix-Sequences}\label{subsec:glt_algebra}
Let $r \geq 1$ be a fixed integer. A multilevel $r$-block GLT sequence, or simply a GLT sequence if we do not need to specify $r$, is a special multilevel $r$-block matrix-sequence  equipped with a measurable function $ \kappa :[0,1]^d \times[-\pi,\pi]^d \to \mathbb{C}^{r\times r}, \quad d \geq 1,$
called the symbol. The symbol is essentially unique, in the sense that if $\kappa, \xi$ are two symbols of the same 
GLT sequence, then $\kappa = \xi$ a.e. We write \( \{A_n\}_n \sim_{\mathrm{GLT}} \kappa \) to denote that $\{A_n\}_n$ is a GLT 
sequence with symbol $\kappa$.\\
It can be proven that the set of multilevel block GLT sequences is the $*$-algebra 
generated by the three classes of sequences defined in Section \ref{sec:zero},\ref{sec:multi},\ref{sec:Block}: zero-distributed, 
multilevel block Toeplitz, and block diagonal sampling matrix sequences. The GLT class 
satisfies several algebraic and topological properties that are treated in detail in \cite{Barbarino2020a, Barbarino2020b, garoni2017, garoni2018}. Here, we focus on the main operative properties listed below that represent a complete characterization of GLT sequences, equivalent to the full constructive definition.

\subsection*{GLT Axioms}

\begin{itemize}
    \item \textbf{GLT 1.} If $\{A_{\bm{n}}\}_{\bm{n}} \sim_{\mathrm{GLT}} \kappa$ then $\{A_{\bm{n}}\}_{\bm{n}} \sim_\sigma \kappa$ in the sense of Definition \ref{def:sing and eig}, with $t = 2d$ and $D = [0, 1]^d \times [-\pi, \pi]^d$. Moreover, if each $A_{\bm{n}}$ is Hermitian, then $\{A_{\bm{n}}\}_{\bm{n}} \sim_\lambda \kappa$, again in the sense of Definition \ref{def:sing and eig} with $t = 2d$.
    \item \textbf{GLT 2.} We have
    \begin{itemize}
        \item $\{T_{\bm{n}}(f)\}_{\bm{n}} \sim_{\mathrm{GLT}} \kappa(\bm{x}, \bm{\theta}) = f(\bm{\theta)}$ if $f : [-\pi, \pi]^d \to \mathbb{C}^{r \times r}$ is in $L^1([- \pi, \pi]^d)$;
        \item $\{D_{\bm{n}}(a)\}_{\bm{n}} \sim_{\mathrm{GLT}} \kappa(\bm{x}, \bm{\theta}) = a(\bm{x})$ if $a : [0, 1]^d \to \mathbb{C}^{r \times r}$ is Riemann-integrable;
        \item $\{Z_{\bm{n}}\}_{\bm{n}} \sim_{\mathrm{GLT}} \kappa(\bm{x}, \bm{\theta}) = O_r$ if and only if $\{Z_{\bm{n}}\}_{\bm{n}} \sim_\sigma 0$.
    \end{itemize}

    \item \textbf{GLT 3.} If $\{A_{\bm{n}}\}_{\bm{n}} \sim_{\mathrm{GLT}} \kappa$ and $\{B_{\bm{n}}\}_{\bm{n}} \sim_{\mathrm{GLT}} \xi$, then:
    \begin{itemize}
        \item $\{A_{\bm{n}}^*\}_{\bm{n}} \sim_{\mathrm{GLT}} \kappa^*$;
        \item $\{\alpha A_{\bm{n}} + \beta B_{\bm{n}}\}_{\bm{n}} \sim_{\mathrm{GLT}} \alpha \kappa + \beta \xi$ for all $\alpha, \beta \in \mathbb{C}$;
        \item $\{A_{\bm{n}} B_{\bm{n}}\}_{\bm{n}} \sim_{\mathrm{GLT}} \kappa \xi$;
        \item $\{A_{\bm{n}}^\dagger\}_{\bm{n}} \sim_{\mathrm{GLT}} \kappa^{-1}$, provided that $\kappa$ is invertible almost everywhere.
    \end{itemize}
    \item \textbf{GLT 4.} \(\{A_{\bm{n}}\}_{\bm{n}} \sim_{\mathrm{GLT}} \kappa\) if and only if there exist \(\{B_{{\bm{n}},j}\}_{\bm{n}} \sim_{\text{GLT}} \kappa_j\) such that \(\{\{B_{{\bm{n}},j}\}_{\bm{n}}\}_j \xrightarrow{\text{a.c.s.wrt }j} \{A_{\bm{n}}\}_{\bm{n}}\) and \(\kappa_j \rightarrow \kappa\) in measure.
    \item \textbf{GLT 5.} If $\{A_{\bm{n}}\}_n \sim_{\mathrm{GLT}} \kappa$ and $A_{\bm{n}} = X_{\bm{n}} + Y_{\bm{n}}$, where
\begin{itemize}
    \item every $X_{\bm{n}}$ is Hermitian,
    \item $ ||X_{\bm{n}}||, \, ||Y_{\bm{{n}}}|| \leq C$ for some constant $C$ independent of $\bm{{n}}$,
    \item ${\nu}(\bm{n})^{-1} \|Y_{\bm{n}}\|_1 \to 0$,
\end{itemize}
then $\{A_{\bm{n}}\}_n \sim_\lambda \kappa$.
   \item \textbf{GLT 6.}  If $\{A_{\bm{n}}\}_n \sim_{\mathrm{GLT}} \kappa$ and each $A_{\bm{n}}$ is Hermitian, then $\{f(A_{\bm{n}})\}_n \sim_{\mathrm{GLT}} f(\kappa)$ for every continuous function $f : \mathbb{C} \to \mathbb{C}$.

\end{itemize}
Note that, by \textbf{GLT 1}, it is always possible to obtain the singular value distribution from the GLT symbol, while the eigenvalue distribution can only be deduced either if the involved matrices are Hermitian or the related matrix-sequence is quasi-Hermitian in the sense of \textbf{GLT 5}.

\section{Geometric Mean of GLT matrix-sequences}\label{sec GM}

We start by recalling Theorem 4 in \cite{ahmad2025matrix}, which generalizes Theorem 10.2 in \cite{garoni2017}.

\begin{theorem}{\rm \cite[Theorem 4]{ahmad2025matrix}}\label{th:two - r=1,d general}
Let $r=1$ and $d\ge 1$.  Suppose \(\{A_{\bm{n}}\}_{\bm{n}} \sim_{\mathrm{GLT}} \kappa\) and \(\{B_{\bm{n}}\}_{\bm{n}} \sim_{\mathrm{GLT}} \xi\), where \(A_{\bm{n}}, B_{\bm{n}} \in \mathcal{P}_{\nu(\bm{n})}\) for every multi-index $\bm{n}$,{ with $\mathcal{P}_{\bm{(n)}}$ denoting the set of Hermitian positive definite (HPD) matrices of multi-index $\bm{n}$}. Assume that at least one between \(\kappa\) and \(\xi\) is nonzero almost everywhere. Then
\begin{equation}\label{r=1, d general GLT}
\{G(A_{\bm{n}}, B_{\bm{n}})\}_{\bm{n}} \sim_{\mathrm{GLT}}(\kappa \xi)^{1/2}
\end{equation}
and
\begin{equation}\label{r=1, d general distributions}
\{G(A_{\bm{n}}, B_{\bm{n}})\}_{\bm{n}} \sim_{\sigma, \lambda} (\kappa \xi)^{1/2}.
\end{equation}
\end{theorem}

Using the topological a.c.s. notion, we now show that the thesis of the previous theorem holds without assuming that at least one between $\kappa$ and $\xi$ is invertible almost everywhere (a.e). 

\begin{theorem}\label{theorem 1}
Let $r=1$ and $d\ge 1$. 
Assume \(\{A_{\bm{n}}\}_{\bm{n}} \sim_{\mathrm{GLT}} \kappa\) and \(\{B_{\bm{n}}\}_{\bm{n}} \sim_{\mathrm{GLT}} \xi\), where \(A_{\bm{n}}, B_{\bm{n}} \in \mathcal{P}_{\nu(\bm{n})}\) for every multi-index $\bm{n}$.  Then
\[
\{G(A_{\bm{n}}, B_{\bm{n}})\}_{\bm{n}} \sim_{\mathrm{GLT}} (\kappa \xi)^{1/2},
\]
\[
\{G(A_{\bm{n}}, B_{\bm{n}})\}_{\bm{n}} \sim_{\sigma, \lambda} (\kappa \xi)^{1/2}.
\]
\end{theorem}

\begin{proof}
We have
\begin{equation}\label{kappa equ}
    \{A_{\bm{n}}\}_{\bm{n}} \sim_{\mathrm{GLT}} \kappa
\end{equation}
\begin{equation}\label{eta equ}
    \{B_{\bm{n}}\}_{\bm{n}} \sim_{\mathrm{GLT}} \xi
\end{equation}
where $\kappa, \xi: [0,1]^d \times [-\pi,\pi]^d \to \mathbb{C}$ and $A_{\bm{n}}, B_{\bm{n}}$ are Hermitian positive definite matrices $\forall \bm{n}$. 
Therefore $\kappa, \xi: [0,1]^d \times [-\pi,\pi]^d \to \mathbb{R}^+_0$ almost everywhere: furthermore, we assume that $ \mu_{2d}(\kappa \equiv 0) > 0$ and $\mu_{2d}(\xi \equiv 0) > 0$, so that the hypotheses of \cite[Theorem 10.2]{garoni2017} for $d=1$ and of Theorem \ref{th:two - r=1,d general} for $d>1$ are violated. 
Let $\varepsilon > 0$, let $\kappa_\varepsilon = \kappa + \varepsilon$ and let $A_{\bm{n}, \varepsilon} = A_{\bm{n}} + \varepsilon I_{{\nu}(\bm{n})}$. Since $\{ \varepsilon I_{{\nu}(\bm{n})}=T_{\bm{n}}(\varepsilon) \}_{\bm{n}} \sim_{\mathrm{GLT}}\varepsilon$ by the first part of Axiom \textbf{GLT 2.} (the identity is a special Toeplitz matrix with GLT symbol $1$), by exploiting linearity i.e. the second item of Axiom \textbf{GLT 3.}, it follows that
\[ 
\{ A_{\bm{n}, \varepsilon} \}_{\bm{n}} \sim_{\mathrm{GLT}} \kappa_\varepsilon, \quad \kappa_\varepsilon \geq \varepsilon \text{ a.e.} 
\]
since $\kappa, \xi \geq 0$ a.e. due to \eqref{kappa equ} and \eqref{eta equ}.
Hence, we are again in the framework of Theorem \ref{th:two - r=1,d general}. Therefore, by  Theorem \ref{th:two - r=1,d general} we conclude that the sequence of geometric means satisfies the GLT relation
\[ 
\{G(A_{\bm{n},\varepsilon}, B_{\bm{n}}) \}_{\bm{n}} \sim_{\mathrm{GLT}} (\kappa_{\varepsilon} \xi)^{\frac{1}{2}}.
\]
Furthermore, by applying the full $*$-algebra framework of the GLT matrix-sequences and more precisely the third item of Axiom \textbf{GLT 3.} followed by Axiom \textbf{GLT 3.} with the function $f(z)=z^{\frac{1}{4}}$, we obtain
\[ \{ (A_{\bm{n}, \varepsilon} B_{\bm{n}}^{2} A_{\bm{n}, \varepsilon})^{\frac{1}{4}} \}_{\bm{n}} \sim_{\mathrm{GLT}} (\kappa_\varepsilon^2 \xi^{2})^{\frac{1}{4}}=(\kappa_\varepsilon \xi)^{\frac{1}{2}}, \]
which the very same GLT symbol of $ \{G(A_{\bm{n},\varepsilon}, B_{\bm{n}}) \}_{\bm{n}} $. As a result, again by the second item of Axiom \textbf{GLT 3.}, the difference between these sequences satisfies the following asymptotic GLT relation
\begin{equation}\label{reducing without inversion}
\{G(A_{\bm{n}, \varepsilon}, B_{\bm{n}}) - (A_{\bm{n}, \varepsilon} B_{\bm{n}}^{2} A_{\bm{n}, \varepsilon})^{\frac{1}{4}} \}_{\bm{n}} \sim_{{\mathrm{GLT}} } 0.
\end{equation}
The previous relation is the key step since we have found a new GLT matrix-sequence having the same GLT symbol as the geometric mean 
matrix-sequence $\{G(A_{\bm{n}, \varepsilon}, B_{\bm{n}})\}_{\bm{n}}$, but where no inversion is required.
Now given the structure of the performed operations, by virtue of Definition \ref{def:acs}, the class $ \{\{ (A_{\bm{n}, \varepsilon} B_{\bm{n}}^{2} A_{\bm{n}, \varepsilon})^{\frac{1}{4}} \}_{\bm{n}},  \varepsilon>0\}$ is obviously an a.c.s. for $\{ (A_{\bm{n}} B_{\bm{n}}^2 A_{\bm{n}})^{\frac{1}{4}} \}_{\bm{n}} $. As a consequence, by invoking (\ref{reducing without inversion}), also the class $ \{\{ G(A_{\bm{n}, \varepsilon}, B_{\bm{n}}) \}_{\bm{n}},  \varepsilon>0\} $
is an a.c.s for $\{ (A_{\bm{n}} B_{\bm{n}}^2 A_{\bm{n}})^{\frac{1}{4}} \}_{\bm{n}} $
with all the involved sequences being GLT matrix-sequences, with symbols $ (\kappa_\varepsilon \xi)^{\frac{1}{2}}, (\kappa\xi)^{\frac{1}{2}}$, $\varepsilon>0$.
Therefore, since  $ \exists \lim_{\varepsilon \to 0} (\kappa_\varepsilon \xi)^{\frac{1}{2}}= (\kappa\xi)^{\frac{1}{2}}$, by using the powerful Theorem \ref{th:fundamental acs}, we deduce that $ \exists \lim_{\varepsilon \to 0} \text{ (a.c.s.) } \{\{ G(A_{\bm{n}, \varepsilon}, B_{\bm{n}}) \}_{\bm{n}}, \varepsilon \}= \{ (A_{\bm{n}} B_{\bm{n}}^2 A_{\bm{n}})^{\frac{1}{4}} \}_{\bm{n}} \sim_{\mathrm{GLT}} (\kappa\xi)^{\frac{1}{2}}$ and this limit coincides with $ \{ G(A_{\bm{n}}, B_{\bm{n}}) \}_{\bm{n}}$ in the a.c.s. topology, that is, 
$\{G(A_{\bm{n}}, B_{\bm{n}}) - (A_{\bm{n}} B_{\bm{n}}^2 A_{\bm{n}})^{\frac{1}{4}} \}_{\bm{n}}\sim_{{\mathrm{GLT}}} 0.$
Finally $\{(A_{\bm{n}} B_{\bm{n}}^2 A_{\bm{n}})^{\frac{1}{4}} \}_{\bm{n}}\sim_{{\mathrm{GLT}}} (\kappa\xi)^{\frac{1}{2}}$ and hence
\[
\{G(A_{\bm{n}}, B_{\bm{n}})\}_{\bm{n}} \sim_{\mathrm{GLT}} (\kappa \xi)^{1/2},
\]
so that, by Axiom \textbf{GLT 1.}, we infer
\[
\{G(A_{\bm{n}}, B_{\bm{n}})\}_{\bm{n}} \sim_{\sigma, \lambda} (\kappa \xi)^{1/2}.
\]
\end{proof}

\begin{remark}[Intuition on the generalization]
\label{gen vs barrier}
One of the basic but key ingredients of the previous proof is that scalar-valued functions commute. Hence, it is reasonable to expect that the same proof also works in the case of $d$-level $r$-block GLT matrix-sequences, when assuming that the symbols commute. We collect the result in Theorem \ref{theorem 2} whose proof follows verbatim that of Theorem \ref{theorem 1}, except for minimal changes. On the other hand, when both $\kappa$ and $\xi$ are not invertible and do not commute, the expression $G(\kappa,\xi)$ is not well defined. In fact, we could replace the inversion with the standard pseudoinversion, which is denoted as $X^+$ if $X$ is any complex matrix. However, we stress that the expression
$\kappa^{\frac{1}{2}} ([\kappa^{\frac{1}{2}}]^+ \xi [\kappa^{\frac{1}{2}}]^+)^{\frac{1}{2}} \kappa^{\frac{1}{2}}$ is not the same as $\xi^{\frac{1}{2}} ([\xi^{\frac{1}{2}}]^+ \kappa [\xi^{\frac{1}{2}}]^+)^{\frac{1}{2}} \xi^{\frac{1}{2}}$ in general and this is a serious indication that the commutation between the GLT symbols is essential.
\end{remark}

We first recall the result proven in \cite{ahmad2025matrix}, concerning the general GLT setting with $d,r \ge 1$ and when at least one of the involved GLT symbols is invertible almost everywhere.

\begin{theorem}{\rm \cite[Theorem 5]{ahmad2025matrix}} \label{th:two - r,d general}
Let $r,d\ge 1$. 
Suppose \(\{A_{\bm{n}}\}_{\bm{n}} \sim_{\mathrm{GLT}} \kappa\) and \(\{B_{\bm{n}}\}_{\bm{n}} \sim_{\mathrm{GLT}} \xi\), where \(A_{\bm{n}}, B_{\bm{n}} \in \mathcal{P}_{\nu(\bm{n})}\) for every multi-index $\bm{n}$. Assume that at least one of the minimal eigenvalues of \(\kappa\) and the minimal eigenvalue of \(\xi\) is nonzero almost everywhere. Then
\begin{equation}\label{r,d general GLT}
\{G(A_{\bm{n}}, B_{\bm{n}})\}_{\bm{n}} \sim_{\mathrm{GLT}} G(\kappa, \xi)
\end{equation}
and
\begin{equation}\label{r,d general distributions}
\{G(A_{\bm{n}}, B_{\bm{n}})\}_{\bm{n}} \sim_{\sigma, \lambda} G(\kappa,\xi).
\end{equation}
Furthermore $G(\kappa,\xi)=(\kappa \xi)^{1/2}$ whenever the GLT symbols  \(\kappa\) and \(\xi\) commute.
\end{theorem}

\begin{theorem}\label{theorem 2}
Let $r>1$ and $d\ge 1$. 
Assume that  \(\{A_{\bm{n}}\}_{\bm{n}} \sim_{\mathrm{GLT}} \kappa\) and \(\{B_{\bm{n}}\}_{\bm{n}} \sim_{\mathrm{GLT}} \xi\), where \(A_{\bm{n}}, B_{\bm{n}} \in \mathcal{P}_{\nu(\bm{n})}\) for every multi-index $\bm{n}$, with $\mathcal{P}_{\bm{(n)}}$ denoting the set of Hermitian positive definite (HPD) matrices of size $r\bm{n}$. Under the assumption that $\kappa$ and $\xi$ commute we infer 
\[
\{G(A_{\bm{n}}, B_{\bm{n}})\}_{\bm{n}} \sim_{\mathrm{GLT}} (\kappa \xi)^{1/2},
\]
\[
\{G(A_{\bm{n}}, B_{\bm{n}})\}_{\bm{n}} \sim_{\sigma, \lambda} (\kappa \xi)^{1/2},
\]
\end{theorem}
\begin{proof}
Since $\{A_{\bm{n}}\}_{\bm{n}} \sim_{\mathrm{GLT}} \kappa$, $\{B_{\bm{n}}\}_{\bm{n}} \sim_{\mathrm{GLT}} \xi$, we deduce that $\kappa, \xi: [0,1]^d \times [-\pi,\pi]^d \to \mathbb{C}^{r\times r}$ are Hermitian nonnegative definite, while, by the assumptions, $A_{\bm{n}}, B_{\bm{n}}$ are Hermitian positive definite matrices $\forall \bm{n}$. 
Therefore, the minimal eigenvalue of $\kappa$ and the minimal eigenvalues of $\xi$ are nonnegative almost everywhere. Here, we assume that $ \mu_{2d}(\lambda_{\min}(\kappa) \equiv 0) > 0$ and $\mu_{2d}(\lambda_{\min}(\xi) \equiv 0) > 0$, in such a way that the hypotheses of Theorem \ref{th:two - r,d general} are violated. 
Let $\varepsilon > 0$, let $\kappa_\varepsilon = \kappa + \varepsilon I_r$, $I_r$ being the identity of size $r$, and let $A_{\bm{n}, \varepsilon} = A_{\bm{n}} + \varepsilon I_{\mathcal{P}(\bm{n})}$. Since $\{ \varepsilon I_{\mathcal{P}(\bm{n})}=T_{\bm{n}}(\varepsilon I_r) \}_{\bm{n}} \sim_{\mathrm{GLT}} \varepsilon I_r$ by the first part of Axiom \textbf{GLT 2.} (the identity $I_{r{\nu}(\bm{n})}$ is a special multilevel block Toeplitz matrix with GLT symbol $I_r$), by exploiting linearity i.e. the second item of Axiom \textbf{GLT 3.}, it follows that
\[ 
\{ A_{\bm{n}, \varepsilon} \}_{\bm{n}} \sim_{\mathrm{GLT}} \kappa_\varepsilon, \quad \kappa_\varepsilon \geq \varepsilon I_r \text{ a.e.} 
\]
since $\kappa, \xi$ are both nonnegative definite a.e.
Hence, by  Theorem \ref{th:two - r,d general}, we have $\{G(A_{\bm{n},\varepsilon}, B_{\bm{n}}) \}_{\bm{n}} \sim_{\mathrm{GLT}} (\kappa_{\varepsilon} \xi)^{\frac{1}{2}}$, 
since the commutation between $\kappa$ and $\xi$ implies the commutation between $\kappa_\varepsilon = \kappa + \varepsilon I_r$ and $\xi$. 
By exploiting the $*$-algebra features of the GLT matrix-sequences, and specifically the third item of Axiom \textbf{GLT 3.} followed by Axiom \textbf{GLT 3.} with the function $f(z)=z^{\frac{1}{4}}$, we obtain
$ \{ (A_{\bm{n}, \varepsilon} B_{\bm{n}}^{2} A_{\bm{n}, \varepsilon})^{\frac{1}{4}} \}_{\bm{n}} \sim_{\mathrm{GLT}} (\kappa_\varepsilon^2 \xi^{2})^{\frac{1}{4}}=(\kappa_\varepsilon \xi)^{\frac{1}{2}}$, 
which is the very same GLT symbol of $ \{G(A_{\bm{n},\varepsilon}, B_{\bm{n}}) \}_{\bm{n}}$. As a result, again by the second item of Axiom \textbf{GLT 3.}, the difference between these sequences satisfies the following asymptotic relation
\begin{equation}\label{reducing without inversion}
\{G(A_{\bm{n}, \varepsilon}, B_{\bm{n}}) - (A_{\bm{n}, \varepsilon} B_{\bm{n}}^{2} A_{\bm{n}, \varepsilon})^{\frac{1}{4}} \}_{\bm{n}} \sim_{{\mathrm{GLT}} } 0.
\end{equation}
In other words, we have written a new GLT matrix-sequence having the same symbol as the geometric mean 
matrix-sequence $\{G(A_{\bm{n}, \varepsilon}, B_{\bm{n}})\}_{\bm{n}}$, but where no inversion is required.
By virtue of Definition \ref{def:acs}, the class $ \{\{ (A_{\bm{n}, \varepsilon} B_{\bm{n}}^{2} A_{\bm{n}, \varepsilon})^{\frac{1}{4}} \}_{\bm{n}},  \varepsilon>0\}$ is an a.c.s. for $\{ (A_{\bm{n}} B_{\bm{n}}^2 A_{\bm{n}})^{\frac{1}{4}} \}_{\bm{n}} $. Therefore, by (\ref{reducing without inversion}), also the class $ \{\{ G(A_{\bm{n}, \varepsilon}, B_{\bm{n}}) \}_{\bm{n}},  \varepsilon>0\} $
is an a.c.s for $\{ (A_{\bm{n}} B_{\bm{n}}^2 A_{\bm{n}})^{\frac{1}{4}} \}_{\bm{n}} $
with all the involved sequences being GLT matrix-sequences, with symbols $ (\kappa_\varepsilon \xi)^{\frac{1}{2}}, (\kappa\xi)^{\frac{1}{2}}$, $\varepsilon>0$. Since  $ \exists \lim_{\varepsilon \to 0} (\kappa_\varepsilon \xi)^{\frac{1}{2}}= (\kappa\xi)^{\frac{1}{2}}$, Theorem \ref{th:fundamental acs} implies $ \exists \lim_{\varepsilon \to 0} \text{ (a.c.s.) } \{\{ G(A_{\bm{n}, \varepsilon}, B_{\bm{n}}) \}_{\bm{n}}, \varepsilon \}= \{ (A_{\bm{n}} B_{\bm{n}}^2 A_{\bm{n}})^{\frac{1}{4}} \}_{\bm{n}} \sim_{\mathrm{GLT}} (\kappa\xi)^{\frac{1}{2}}$. It is now clear that this limit coincides with $ \{ G(A_{\bm{n}}, B_{\bm{n}}) \}_{\bm{n}} $
in the a.c.s. topology, that is, 
$\{G(A_{\bm{n}}, B_{\bm{n}}) - (A_{\bm{n}} B_{\bm{n}}^2 A_{\bm{n}})^{\frac{1}{4}} \}_{\bm{n}}\sim_{{\mathrm{GLT}}} 0.$
In this manner $\{(A_{\bm{n}} B_{\bm{n}}^2 A_{\bm{n}})^{\frac{1}{4}} \}_{\bm{n}}\sim_{{\mathrm{GLT}}} (\kappa\xi)^{\frac{1}{2}}$. Consequently
$\{G(A_{\bm{n}}, B_{\bm{n}})\}_{\bm{n}} \sim_{\mathrm{GLT}} (\kappa \xi)^{1/2}$, so that, by Axiom \textbf{GLT 1.}, we finally obtain
$\{G(A_{\bm{n}}, B_{\bm{n}})\}_{\bm{n}} \sim_{\sigma, \lambda} (\kappa \xi)^{1/2}$.
\end{proof}

\section{Numerical experiments}\label{Num_Exp}

As it is well known, many localization, extremal, and distribution results hold when $d$-level, $r$-block Toeplitz matrix-sequences are considered and these results are somehow summarized in specific analytic features of the generating function of the corresponding matrix-sequence. In turn, the generating function is also the $d$-variate, $r\times r$ matrix-valued GLT symbol of the $d$-level, $r$-block Toeplitz matrix-sequence. 

While the distribution results are also valid for general $d$-level, $r$-block GLT matrix-sequences, this is no longer true in general for the extremal behavior and for the localization results, unless we add supplementary assumptions, like the request that the matrix-sequence is obtained via a matrix-valued linear positive operator (LPO) \cite{SerraCapizzanoTilli-LPO,LPO-rev}. We observe that the geometric mean can be seen as a monotone operator with respect to its two variables, and the monotonicity is implied when we consider an LPO, even if the converse is not true. Hence, it is also interesting to verify which properties are maintained by a geometric mean of two $d$-level, $r$-block GLT matrix-sequences in terms of its GLT symbol when it exists.

According to the previous discussion, the rest of the section considers numerical experiments in the following directions.

\begin{itemize}
\item Validation of the distribution results in the commuting setting as in Theorem \ref{theorem 1} and in Theorem \ref{theorem 2}.
\item Connections of the previous results with the notion of Toeplitz and GLT momentary symbols and the extremal behavior compared to the GLT symbol. 
\item Evidence that Theorem \ref{theorem 1} and Theorem \ref{theorem 2} are maximal, by taking GLT matrix-sequences with noncommuting symbols which are both not invertible a.e. 
\end{itemize}
  
\subsection{Validation of the distribution results} \label{ssec:num1}
  
\subsubsection{Example 1 }
Let $ d = 1 $, and consider the following two matrix sequences $
A_n = D_n(a) + \frac{1}{n^4} I_n$ and  $ B_n = T_n(3 + 2\cos(\theta)),$
 where $T_n(\cdot) $ denotes the Toeplitz operator for $d=1 $, as introduced in Section \ref{sec:multi}, and $D_n(a) $ represents the diagonal matrix generated by the continuous function
\[
a(x) =
\begin{cases}
0, & x \in [0, \frac{1}{2}), \\
1, & x \in [\frac{1}{2},1].
\end{cases}
\]
The geometric mean of these two sequences, $\{ G(A_n, B_n) \} $, is known to satisfy the GLT relation:
\[
\left\{ G\left(A_n, B_n\right) \right\}_n \sim_{\mathrm{GLT}} \sqrt{a(x)(3+2\cos(\theta))} 
\]
where the corresponding GLT symbols are given explicitly by
$
\kappa = a(x)$ and $ \xi=3+2\cos(\theta).
$
\subsubsection*{Eigenvalue distribution}
We numerically analyze the spectral behavior of the geometric mean $G(A_n, B_n) $ from \textbf{Example 1}. The eigenvalues of this geometric mean are computed for various increasing dimensions $n $ and compared against uniformly sampled points from the GLT symbol $\sqrt{a(x)(3+2\cos(\theta))} $. In \textbf{Figure 1}, numerical results strongly indicate that the GLT symbol accurately characterizes the eigenvalue distribution of the geometric mean. As $ n $ grows, the eigenvalue distribution closely matches the GLT symbol, giving evidence of the theoretical findings.
\begin{figure}[H]
   \begin{minipage}{0.52\textwidth}
     \centering
     \includegraphics[width=\linewidth, height=6cm]{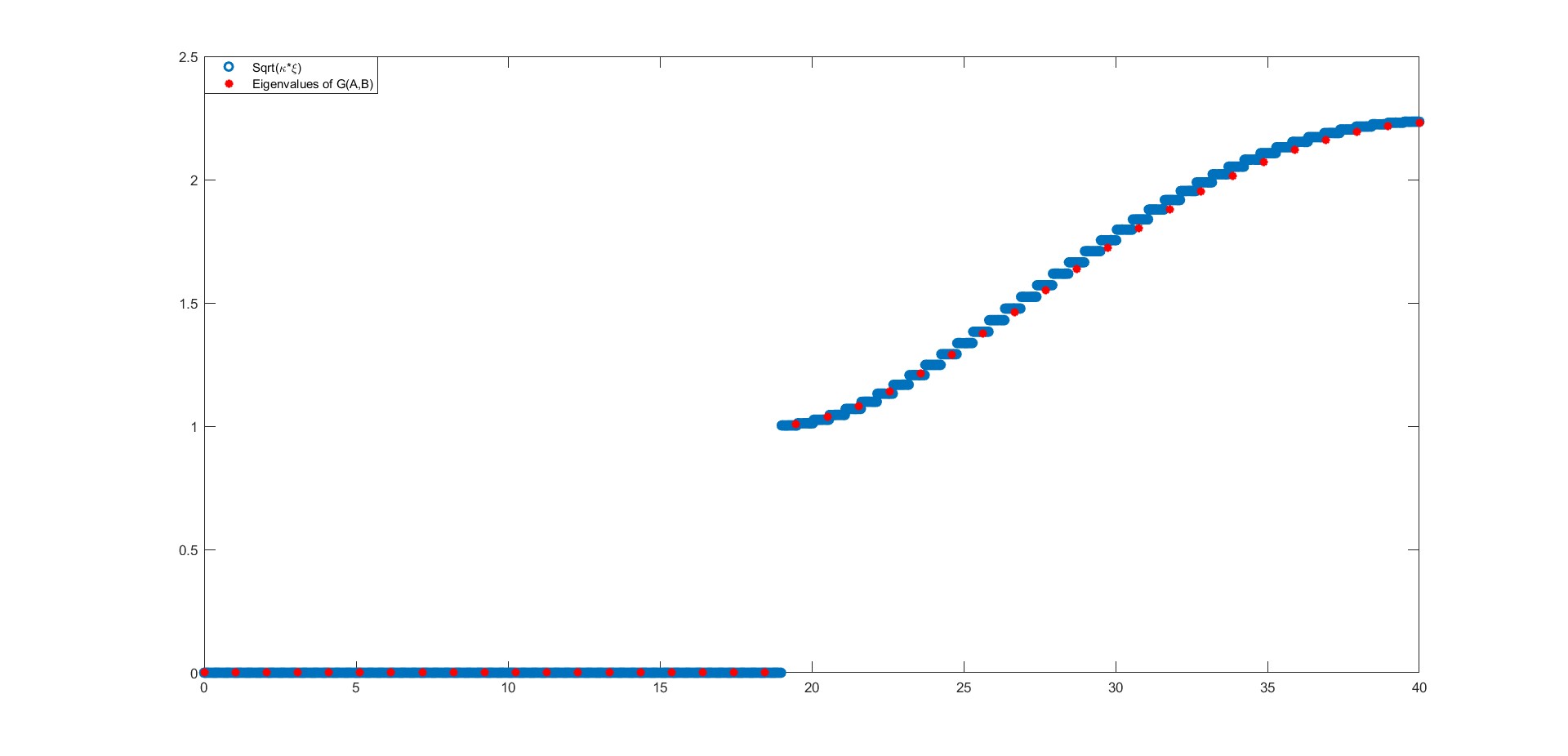} 
     \captionsetup{labelformat=empty}
     \caption{(a) n=40}
   \end{minipage}\hfill
   \begin{minipage}{0.52\textwidth}
     \centering
     \includegraphics[width=\linewidth, height=6cm]{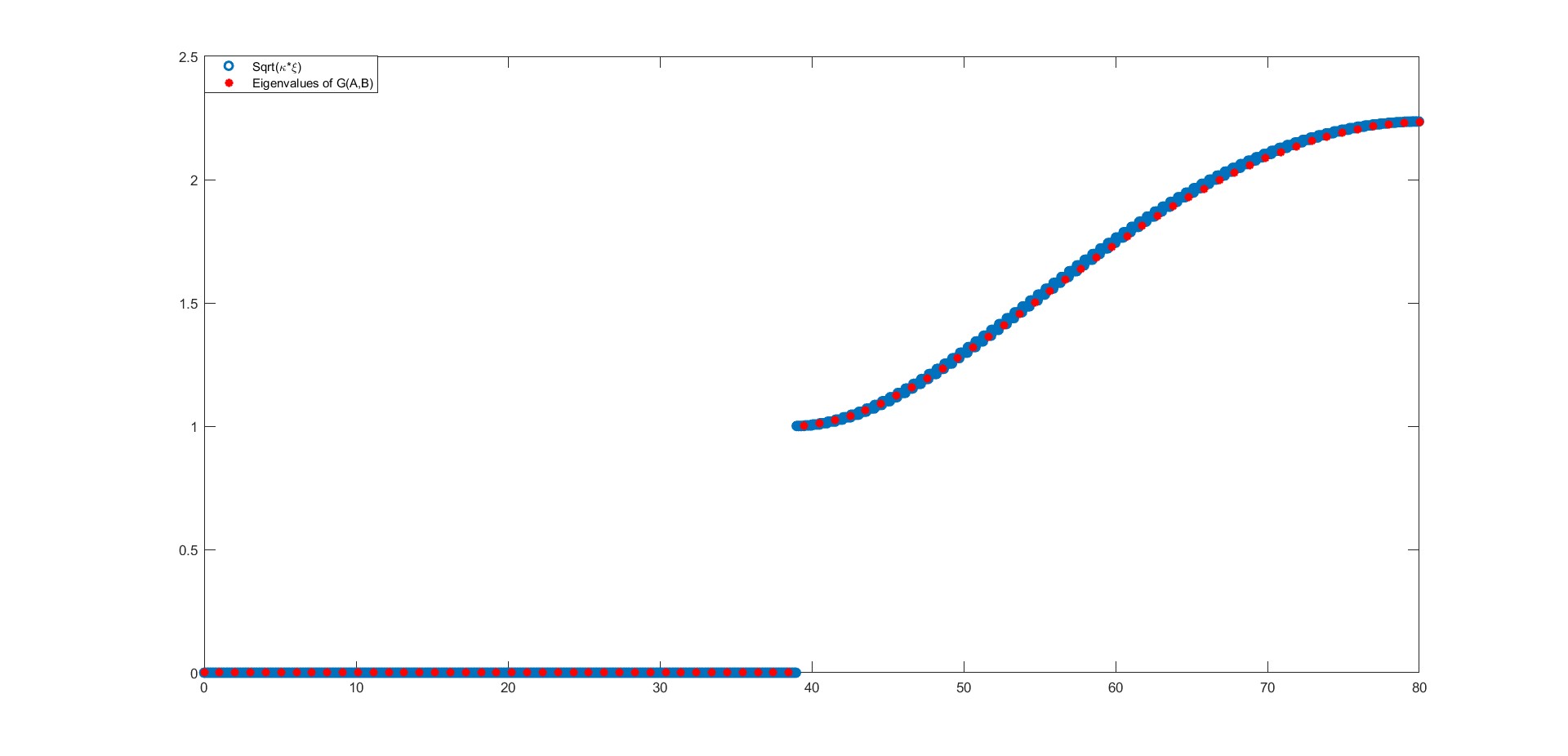} 
     \captionsetup{labelformat=empty}
     \caption{(b) n=80}
   \end{minipage}
\end{figure}
\quad \quad  \quad \textbf{Figure 1}: Comparison between the symbol $(\kappa\xi)^{\frac{1}{2}} $ and $\text{eig}(G(A_{n},B_{n}) $.

\subsection{Distribution results and momentary symbols} \label{ssec:num2}

\subsubsection{Example 2 }
Let $d = 1 $, and define the sequences $
A_n = D_n(a) + \frac{1}{n^4} I_n,$  $ C_n = T_n(3 + 2\cos(\theta))
$ and \\$
B_n = \left(D_n(1 - a) + \frac{1}{n^4} I_n \right) C_n \left(D_n(1 - a) + \frac{1}{n^4} I_n \right),
$ where \( D_n(a) \) is the diagonal matrix generated by the piecewise continuous function \( a(x) \) given in \textbf{Example 1}. Then, the geometric mean sequence explicitly satisfies the GLT relation:
\[
\left\{ G\left(A_n, B_n\right) \right\}_n \sim_{\mathrm{GLT}} \sqrt{a(x)(1-a(x))(3+2\cos(\theta))} = a(x)(1-a(x)) \sqrt{3+2\cos(\theta)}=0,
\]
where
\[
\kappa = a(x) \quad \xi=(1-a(x))(3+2\cos(\theta)).
\]
The zero-valued symbol arises naturally due to the definition of the piecewise function \( a(x) \):
\begin{itemize}
    \item For \( x \in [0, \frac{1}{2}) \), we have \( a(x) = 0 \). Thus, the term \( a(x)(1 - a(x)) \) becomes zero.
    \item For \( x \in [\frac{1}{2},1] \), we have \( a(x) = 1 \). In this interval, the factor \( (1 - a(x)) \) becomes zero, making \( a(x)(1 - a(x)) = 0 \).
\end{itemize}

\subsubsection*{Eigenvalue distribution}
The eigenvalue distribution of \( G(A_n, B_n) \) is analyzed numerically for increasing matrix dimensions \( n \). Since the GLT symbol for this example is zero, we expect the eigenvalues to concentrate along the zero line. However, numerical results reveal additional structure:\\

{For smaller values of $ n $, eigenvalues align closely with zero, as predicted by the GLT symbol. However, some eigenvalues appear above this level, forming a secondary symbol. This phenomenon is attributed to {momentary symbols}, as discussed in the introduction (see \cite{bolten2022toeplitz,bolten2023note,new momentary}). As \( n \) increases, these elevated eigenvalues shift upwards—reaching approximately \( 10^{-3} \) for \( n = 40 \) (shown in the \textbf{Figure 2}) and \( 10^{-4} \) for \( n = 80 \).}
\begin{figure}[H]
   \begin{minipage}{0.52\textwidth}
     \centering
     \includegraphics[width=\linewidth, height=6cm]{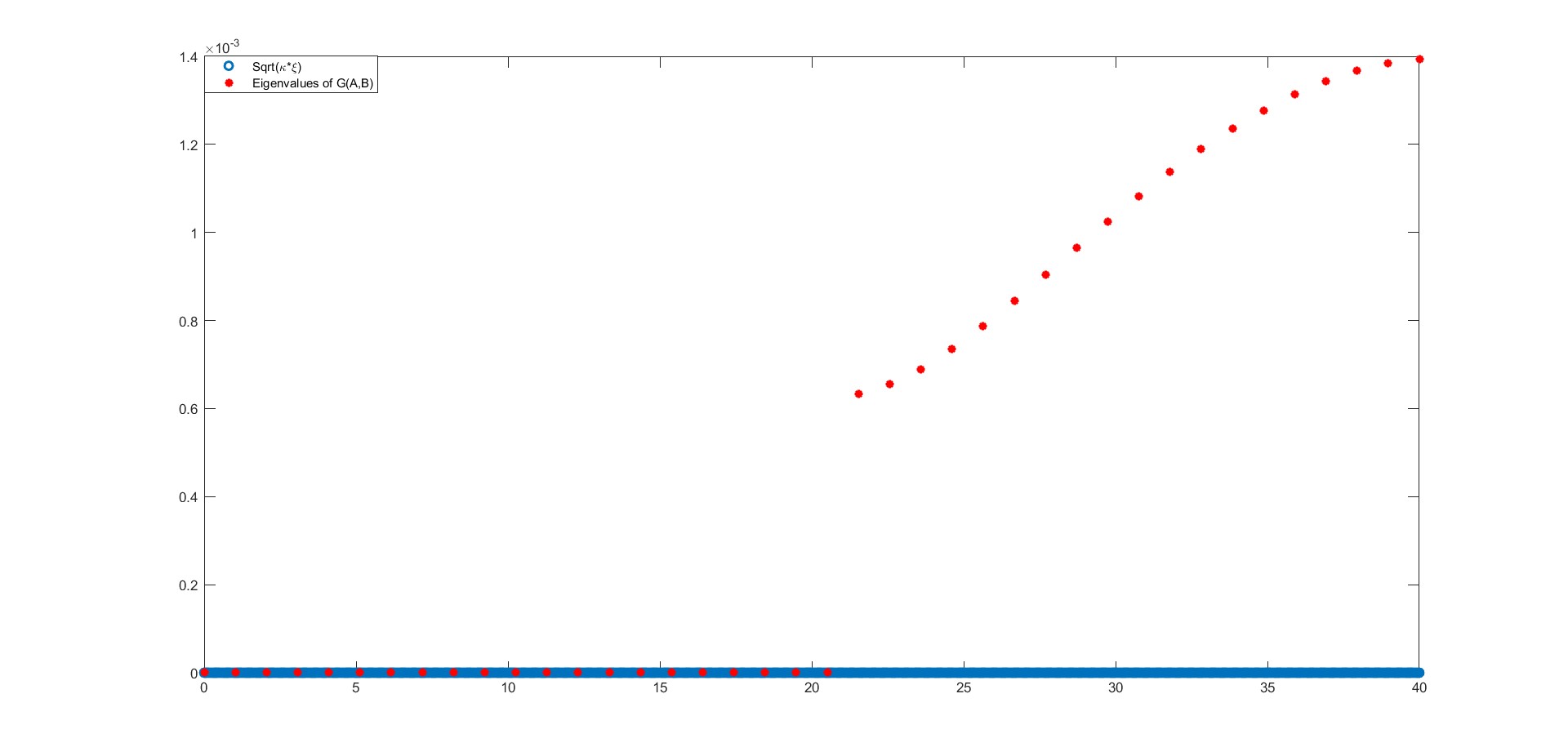} 
     \captionsetup{labelformat=empty}
     \caption{(a) n=40}
   \end{minipage}\hfill
   \begin{minipage}{0.52\textwidth}
     \centering
     \includegraphics[width=\linewidth, height=6cm]{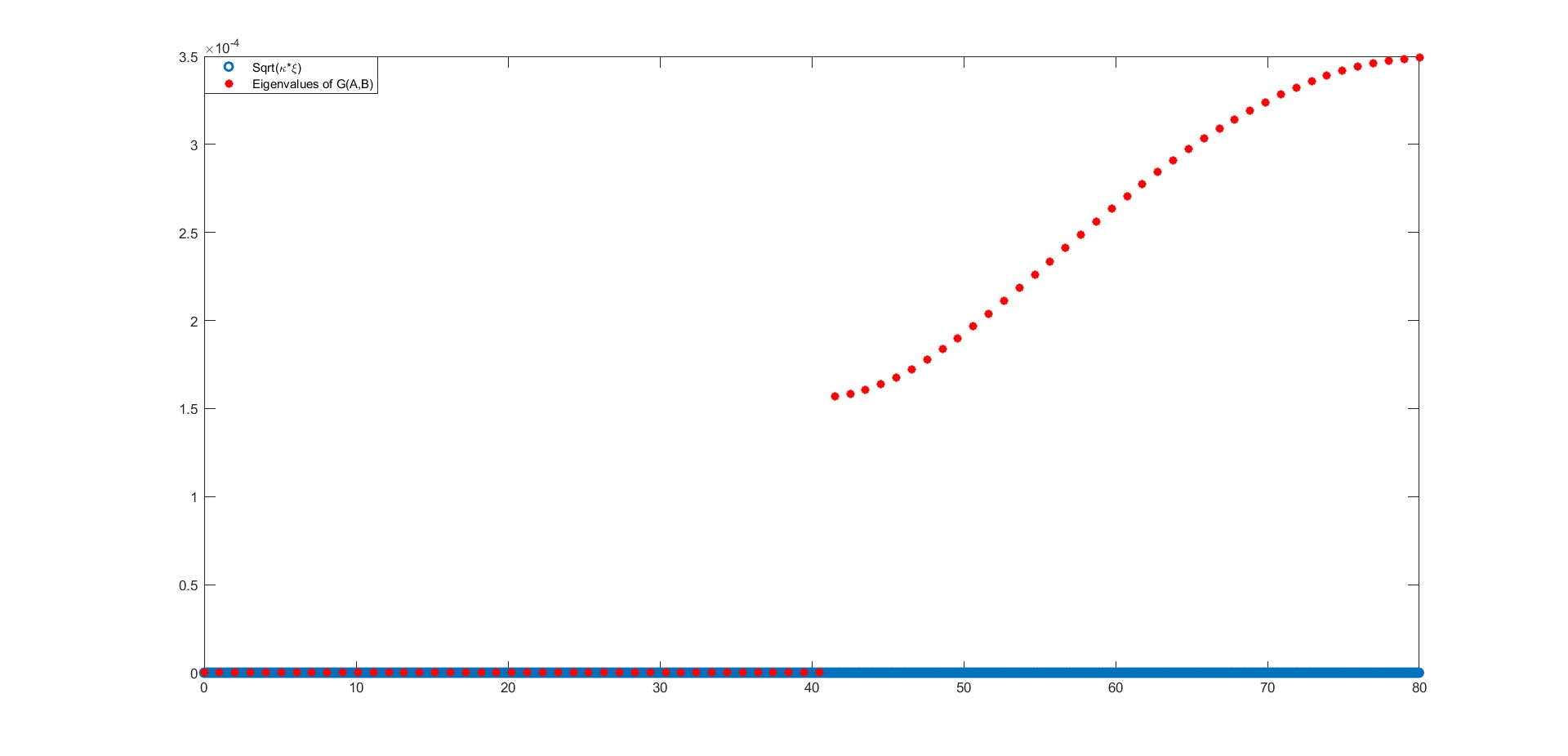} 
     \captionsetup{labelformat=empty}
     \caption{(b) n=80}
   \end{minipage}
\end{figure}
\quad \quad  \quad \textbf{Figure 2}: Comparison between the symbol \( (\kappa\xi)^{\frac{1}{2}} \) and \( \text{eig}(G(A_{n},B_{n}) \).

\subsection{Minimal eigenvalues and conditioning} \label{ssec:num3}

In this section, we analyze the extremal spectral behavior and conditioning of the geometric mean sequence \( \{ G(A_n, B_n) \}_n \), focusing on its dependence on the analytical properties of the corresponding GLT symbol. This approach follows prior studies on extremal eigenvalues in structured matrix settings, particularly in Toeplitz matrices \cite{Bottcher1998,SerraCapizzano1996,SerraCapizzano1998,SerraCapizzanoTilli1999} and block Toeplitz matrices \cite{SerraCapizzano1999a,SerraCapizzano1999b}, as well as variable coefficient differential operators, including multilevel cases with \( d > 1 \) \cite{Vassalos2018,Noutsos2008,extr2-glt,extr2-glt}. In the case of variable coefficient differential operators it is worth noticing that extremal spectral results do not stem from the GLT theory but from a combination of GLT tools and properties which are typical of linear positive operators 
\cite{SerraCapizzanoTilli-LPO,LPO-rev}.

Here, we restrict our attention to the unilevel scalar setting with \( d = r = 1 \), considering \textbf{Example 1} and \textbf{Example 2}.
\subsubsection{Example 1: Minimal Eigenvalue}
\begin{itemize}
    \item $X_{n}=G(A_{n},B_{n})$
    \item Take $n_{j}=40.2^{j}, \quad\quad j=0,1,2,3, $
    \item Compute $\tau_{j}= \lambda_{min}(X_{n}), \quad \quad j=0,1,2,3, $
    \item Compute $\alpha_{j}=\log_{2}(\frac{\tau_{j}}{\tau_{j+1}}), \quad \quad j=0,1,2. $ 
\end{itemize}
\vspace{9pt}

\begin{table}[H]
  \centering
  \caption{Numerical behaviour of the minimal eigenvalue.}
  \label{tab:tau_alpha}
  \begin{tabular}{|c
                  |S[table-format=1.4e+2]
                  |S[table-format=1.4]|}
    \hline
    {$n$} & {$\tau_j$} & {$\alpha_j$}\\
    \hline
    40  & 6.3260e-04 & 2.0132 \\
    80  & 1.5670e-04 & 2.0064 \\
    160 & 3.9000e-05 & 2.0074 \\
    320 & 9.7000e-06 & \multicolumn{1}{c|}{--}     \\
    \hline
  \end{tabular}
\end{table}

As it can be seen, the quantity $\alpha_j$ converges to $2$ as $n$ increases. This is in perfect agreement with the fact that the minimal eigenvalue of $A_n$ converges to zero as $n^{-4}$, while the minimal eigenvalue of $B_n$ converges monotonically from above to $\min \xi=1$. The key point is that the minimal eigenvalue of the geometric mean behaves asymptotically as the geometric mean of the minimal eigenvalues of  $A_n$ and $B_n$, respectively. Similar remarks can be done in the subsequent case regarding Example 2.
 
 The observed numerical evidences are not implied by the theoretical derivations and this is an interesting fact, which deserves to be investigated theoretically in the future.

\subsubsection{Example 2: Minimal Eigenvalue}
\begin{itemize}
    \item $X_{n}=G(A_{n},B_{n})$
    \item Take $n_{j}=40.2^{j}, \quad\quad j=0,1,2,3, $
    \item Compute $\tau_{j}= \lambda_{min}(X_{n}), \quad \quad j=0,1,2,3, $
    \item Compute $\alpha_{j}=\log_{2}(\frac{\tau_{j}}{\tau_{j+1}}), \quad \quad j=0,1,2. $ 
\end{itemize}
\vspace{9pt}
\begin{table}[H]
  \centering
  \caption{Numerical behaviour of the minimal eigenvalue.}
  \label{tab:tau_alpha}
  \begin{tabular}{|c
                  |S[table-format=1.4e+2]
                  |S[table-format=1.4]|}
    \hline
     {$n$} & {$\tau_j$} & {$\alpha_j$} \\
    \hline
    40  & 3.9177e-07 & 4.0003 \\
    80  & 2.4480e-08 & 4.0040 \\
    160 & 1.5250e-09 & 4.0009 \\
    320 & 9.5000e-11 & \multicolumn{1}{c|}{--}      \\
    \hline
  \end{tabular}
\end{table}

%====================================================================
% Inizio della sezione numerica non-commuting. L'introduzione si può rendere più specifica
%====================================================================

\subsection{Numerical study: non‑commuting, rank–deficient symbols}
\label{sec:non_commuting_rank_def}

In this section, we investigate, through numerical tests, the eigenvalue distribution of the geometric mean
\[
   G(A_n,B_n)\;=\;
   A_n^{1/2}\!\bigl(A_n^{-1/2}B_nA_n^{-1/2}\bigr)^{1/2}\!A_n^{1/2},
\]
where $\{A_n\}_n$ and $\{B_n\}_n$ are HPD GLT sequences  with size~$r$ matrix‑valued symbols  
$\kappa,\xi$.  
The GLT sequences are specifically chosen with GLT symbols so that: 

\begin{enumerate}
  \item they are rank‑deficient on a subset of positive Lebesgue measure; and
  \item they do not commute pointwise on a set of positive measure.
\end{enumerate}

These hypotheses clearly set the problem outside the scope of commuting or almost everywhere invertible cases treated in Theorems~\ref{theorem 1} and~\ref{theorem 2}. \\
Our tests provide evidence for two phenomena:
\begin{itemize}
  \item[\textbf{(i)}]
    The support of the asymptotic eigenvalue distribution of $G(A_n,B_n)$ coincides with the intersection set $\operatorname{ess\,Ran}(\kappa)\cap\operatorname{ess\,Ran}(\xi)$. 
  \item[\textbf{(ii)}]  The eigenvalues appear to converge in distribution to a candidate symbol which is rank‑deficient and, in general, distinct from the classical geometric mean $G(\kappa,\xi)$.
\end{itemize}
Clearly, the relation \( G(A_n, B_n) \sim_{\mathrm{GLT}} G(\kappa, \xi) \) cannot hold under hypotheses (1)-(2), and the experiments show the substantial differences in this case compared to the commuting or invertible scenarios. If this case can be studied within the GLT framework, it requires a more technical application of GLT and a.c.s. theory, along with a deeper functional calculus of matrix means.

%%%%%%%%%%%%%%%%%%%%%%%%%%%%%%%%%%%%%%%%%%%%%%%%%%%%%%%%%%%%%%%%%%%%%%%%%%%%%% 
\subsubsection{Setup of the numerical experiments}
\label{sec:NumSetup}

Below, we provide the setup of the numerical tests. \\
We construct four pairs of unilevel GLT sequences $\bigl(\{A_n\}_n,\{B_n\}_n\bigr)$ whose symbols
simultaneously satisfy the hypothesis stated at the beginning of
Section~\ref{sec:non_commuting_rank_def} (rank deficiency and
non‑commutation). They are organized by the property of rank loss:

\begin{itemize}\setlength\itemsep{4pt}
  \item \textbf{Case 1} — each symbol is full rank on a set of positive
        measure;
  \item \textbf{Case 2} — each symbol is rank‑deficient almost everywhere.
\end{itemize}

%---------------------------------------------------------------------------
\medskip\noindent
\textbf{Case 1, Example 1.}

Define
\[
  f(\theta)\; =\; \begin{cases}
               0, & \theta\in[-\pi,0],\\[4pt]
               \theta, & \theta\in(0,\pi],
             \end{cases}
  \qquad
  g(\theta) = f(-\theta),
\]
and set
\[
  F(\theta)\;=\;f(\theta)\otimes
     \begin{bmatrix}2&1\\[2pt]1&2\end{bmatrix},\qquad
  G(\theta)\;=\;g(\theta)\otimes
     \begin{bmatrix}3&1\\[2pt]1&1\end{bmatrix}.
\]
With the $2n\times2n$ Toeplitz matrices $T_n(F)$ and $T_n(G)$, define
\begin{equation}\label{eq:case1_ex1}
  A_n=T_n(F)+\frac{1}{n^{3}}I_{2n},\qquad
  B_n=T_n(G)+\frac{1}{n^{3}}I_{2n}.
\end{equation}

%---------------------------------------------------------------------------
\medskip\noindent
\textbf{Case 1, Example 2.}

Let $\chi_{[-a,a]}$ be the characteristic function of $[-a,a]$, $0<a<\pi$, and set
\[
  f(\theta)\,=\,\chi_{[-1/2,\,1/2]}(\theta),\qquad
  g(\theta)\,=\,\chi_{[-1/4,\,1/4]}(\theta),
\]
\[
  A\;=\;\begin{bmatrix}2&1\\[2pt]1&2\end{bmatrix},\qquad
  B\;=\;\begin{bmatrix}3&1\\[2pt]1&1\end{bmatrix},
\]
\[
  F(\theta)\;=\;f(\theta)\otimes A,\qquad
  G(\theta)\;=\;g(\theta)\otimes B.
\]
Define
\begin{equation}\label{eq:case1_ex2}
  A_n=T_n(F)+\frac{1}{n^{3}}I_{2n},\qquad
  B_n=T_n(G)+\frac{1}{n^{3}}I_{2n}.
\end{equation}

%---------------------------------------------------------------------------
\medskip\noindent
\textbf{Case 2, Example 1.}

With $f(\theta)=2-\cos\theta$, $g(\theta)=3+\cos\theta$ and
rank one blocks
\[
  A\;=\;\begin{bmatrix}1&1\\[2pt]1&1\end{bmatrix},\qquad
  B\;=\;\begin{bmatrix}1&2\\[2pt]2&4\end{bmatrix},
\]
let
\(F(\theta)\;=\;f(\theta)\otimes A\) and
\(G(\theta)\;=\;g(\theta)\otimes B\).
Set
\[
  A_n=T_n(F)+\frac{1}{n^{2}}I_{2n},\qquad
  B_n=T_n(G)+\frac{1}{n^{2}}D_n(b),\qquad
  b(x)=(1+x) I_2.
\]
Here, both symbols are rank $1$ almost everywhere.

%---------------------------------------------------------------------------
\medskip\noindent
\textbf{Case 2, Example 2.}

Define piecewise–linear scalar functions in $[0,1]$:
\[
  a(x)\;=\;\begin{cases}
          1-2x, & 0\le x\le\tfrac12,\\[2pt]
          0, & \tfrac12<x\le1,
        \end{cases}
  \qquad
  b(x)\;=\;\begin{cases}
          0, & 0\le x\le\tfrac13\text{ or } \tfrac23\le x\le1,\\[2pt]
          x-\tfrac13, & \tfrac13<x\le\tfrac12,\\[2pt]
          \tfrac23-x, & \tfrac12<x<\tfrac23,
        \end{cases}
\]
and let $D_n(a)$, $D_n(b)$ be the corresponding sampling matrices.
With $f(\theta)=2+\cos\theta$, $g(\theta)=3+\cos\theta$ and
\[
  A\;=\;\begin{bmatrix}2 & 0 & 1\\[2pt]
  0 & 2 & 1\\[2pt]
  1 & 1 & 1\end{bmatrix},\qquad
  B\;=\;\begin{bmatrix}2 & 1 & 0\\[2pt]
  1 & 1 & 1\\[2pt]
  0 & 1 & 2\end{bmatrix},
\]
define
\(F(\theta)\;=\;f(\theta)\otimes A\) and
\(G(\theta)\;=\;g(\theta)\otimes B\), and set
\begin{align}\label{eq:case2_ex2}
      A_n &\;=\; \bigl(D_n^{1/2}(a)\,\otimes\, I_3\bigr)\,
T_n(F)\, \bigl(D_n^{1/2}(a)\,\otimes\, I_3\bigr)
+\frac{1}{5n}\,I_{3n}, \\
  B_n&\;=\; \bigl(D_n^{1/2}(b)\,\otimes\, I_3\bigr)\,
T_n(G)\,
\bigl(D_n^{1/2}(b)\,\otimes\, I_3\bigr)
+\frac{1}{5n}\,I_{3n}
\end{align}

\medskip
In every example we denote by $\kappa$ and $\xi$ the GLT symbols of
$\{A_n\}_n$ and $\{B_n\}_n$, respectively.

\medskip
%--------------------------------------------------------------------
%--------------------------------------------------------------------
%%%%%%%%%%%%%%%%%%%%%%%%%%%%%%%%%%%%%%%%%%%%%%%%%%%%%%%%%%%%%%%%%%%%%%%%%%%%%%
\subsubsection{A candidate symbol for the geometric mean}
\label{subsubsec:GM_conjecture}

Let $\{A_n\}_n$ and $\{B_n\}_n$ be GLT sequences whose matrix–valued symbols  
\[
  \kappa,\;\xi : [0,1]^d\times[-\pi,\pi]^d \;\longrightarrow\; \mathbb{C}^{\,r\times r}
\]
are positive–semidefinite. \\
For any $\varepsilon>0$ define the strictly positive symbols  
\(
  \kappa_\varepsilon=\kappa+\varepsilon I_r,\;
  \xi_\varepsilon=\xi+\varepsilon I_r.
\)
Since the matrix geometric mean is monotone continuous in each argument, the limit below exists point‑wise {and it is unique} (\cite[Pag.\,3]{ando2004geometric}):

\begin{definition}[Candidate symbol]\label{def:candidate}
For almost every $(x,\theta)$ set
\begin{equation}\label{eq:cand_symbol}
  \widetilde G(\kappa,\xi)(x,\theta)
  :=\lim_{\varepsilon\rightarrow 0}
        G\bigl(\kappa_\varepsilon(x,\theta),\,
               \xi_\varepsilon (x,\theta)\bigr).
\end{equation}
\end{definition}

The matrix $\widetilde G(\kappa,\xi)$ is positive semidefinite and  
\[
  \operatorname{ess \, Ran}\left[\widetilde G(\kappa,\xi)(x,\theta)\right]
  =
  \operatorname{ess \,Ran}(\kappa(x,\theta)) \cap \operatorname{ess \, Ran}(\xi(x,\theta)).
\]
%Writing $W(x,\theta)$ for this intersection, $m=\dim W$, and letting
%$U(x,\theta)\in \mathbb{C}^{r\times m}$ be any orthonormal basis matrix for $W$,
%the compression identity
%\begin{equation}\label{eq:compression_formula}
%  \widetilde G(\kappa,\eta)
%  = U\,G\bigl(U^{*}\kappa\,U,\;U^{*}\eta\,U\bigr)\,U^{*}
%\end{equation}
%holds independently of the choice of $U$.
{Given the Definition \ref{def:candidate} of the candidate symbol, we propose the following conjecture:}
\begin{conjecture}[GLT closure under geometric mean]
\label{conj:GM_glt}
Let $\{A_n\}_n$ and $\{B_n\}_n$ be Hermitian positive definite
GLT matrix-sequences with
\[
     \{A_n\}_n\sim_{\mathrm{GLT}}\kappa,
     \qquad
     \{B_n\}_n\sim_{\mathrm{GLT}}\xi.
\]
Then the geometric mean matrix-sequence is again GLT and
\[
     \{\,G(A_n,B_n)\,\}_n \;\sim_{\operatorname{GLT}}\; \widetilde G(\kappa,\xi),
\]
where $\widetilde G(\kappa,\xi)$ is the candidate symbol
defined in~\eqref{eq:cand_symbol}.
\end{conjecture}

The heuristic justification of the conjecture is the following. For each fixed $\varepsilon>0$ the shifted sequences
$\{A_n+\varepsilon I\}_n$ and $\{B_n+\varepsilon I\}_n$ are GLT sequences symbols
$\kappa_\varepsilon$ and $\xi_\varepsilon$ and their geometric mean sequences satisfy
\(
  \{\,G(A_n+\varepsilon I,B_n+\varepsilon I)\,\}_n
  \sim_{\operatorname{GLT}}
  G(\kappa_\varepsilon,\xi_\varepsilon).
\)
{Properties of monotone functions and the matrix geometric mean axioms described in \cite{ando2004geometric}}, combined with closeness of the GLT *-algebra (Axiom \textbf{GLT 4}) under a.c.s convergence, suggest that letting $\varepsilon\rightarrow0$ the limit of $ \{\,G(A_n+\varepsilon I,B_n+\varepsilon I)\,\}_n$ is well defined in the {a.c.s topology and this limit is  $ \{\,G(A_n,B_n)\,\}_n$, up to zero-distributed perturbations }, leading to the GLT algebraic part of Conjecture~\ref{conj:GM_glt}, a full proof when
$\kappa$ and $\xi$ do not commute point‑wise remains open. Meanwhile, all
numerical evidence below supports the conjecture on the spectral distribution part of our hypothesis. \\ 

%%%%%%%%%%%%%%%%%%%%%%%%%%%%%%%%%%%%%%%%%%%%%%%%%%%%%%%%%%%%%%%%%%%%%%%%%%%%%%
\subsubsection*{Candidate symbols predicted by Conjecture~\ref{conj:GM_glt}}

For the four test pairs introduced in Section \ref{sec:NumSetup}, Conjecture~\ref{conj:GM_glt} predicts
the following GLT symbols for $\{G(A_n,B_n)\}_n$.
Throughout
\(
  \mathbf 1=(1,1,1)^{\mathsf *},\;
  J=\mathbf 1\,\mathbf 1^{\mathsf *}\in\mathbb{C}^{3\times3}.
\)

\paragraph{Case 1, Example 1.}
\[
  \widetilde G(\kappa,\xi)(x,\theta)=0_{2\times2},
\]

\paragraph{Case 1, Example 2.}
\[
  \widetilde G(\kappa,\xi)(x,\theta)=
  \chi_{[-\frac14,\frac14]}(\theta)\;
  \otimes C,
  \qquad
  C \; =\; 
  \frac{1}{6^{1/4}\sqrt{\,2+\sqrt6\,}}
  \begin{bmatrix}
     2\sqrt2+3\sqrt3 & \sqrt2+\sqrt3\\[4pt]
     \sqrt2+\sqrt3   & 2\sqrt2+\sqrt3
  \end{bmatrix}.
\]

\paragraph{Case 2, Example 1.}
\[
  \widetilde G(\kappa,\xi)(x,\theta)=0_{2\times2} 
\]

\paragraph{Case 2, Example 2.}
\[
  \widetilde G(\kappa,\eta)(x,\theta)=
  \sqrt{\,\bigl(2+\cos\theta\bigr)\,\bigl(3+\cos\theta\bigr)\,h(x)\,k(x)}\;
  \otimes J.
\]

\subsubsection{Numerical verification of the Conjecture}

Let
\[
  G_n := G(A_n,B_n)\in\mathbb{C}^{d_n\times d_n},
  \qquad d_n=\text{size}(G_n).
\]
We test whether the empirical eigenvalue distribution of $\{G_n\}_n$
converges to the distribution induced by the
candidate symbol
$\widetilde G(\kappa,\xi)(x,\theta)$ defined in
Section~\ref{subsubsec:GM_conjecture}.
The comparison follows the standard rearrangement strategy.

\begin{itemize}\setlength\itemsep{6pt}

\item[\textbf{1.}] \textbf{Sampling the symbol.}
      Evaluate $\widetilde G(\kappa,\xi)$ on a tensor grid
      $\{x_j\}_{j=1}^{M_x}\!\subset[0,1]$,
      $\{\theta_i\}_{i=1}^{M_\theta}\!\subset[-\pi,\pi]$ with a total number of
      $M_x \times M_\theta\simeq2000$ samples.
      For each point $(x_j,\theta_i)$ compute the
      $r$ eigenvalues of the function; then merge and sort in non-decreasing order all the computed values.
      This yields the empirical quantile function of the eigenvalues.

\item[\textbf{2.}] \textbf{Eigenvalue computation.}
      Matrices $A_n,B_n$ are generated for
      $n\in\{40,80,160,320\}$, and eigenvalues are ordered in non-decreasing order, obtaining the quantile spectral distribution function of the $\{G_n\}_n$. \\
      We also record
      $\lambda_{\min}(G_n)$, $\lambda_{\max}(G_n)$ and the condition
      number~$\mathrm{cond}_2(G_n)$ for later stability analysis.

\item[\textbf{3.}] \textbf{Quantile plot matching.}
      For each $n$ we plot the sorted eigenvalues
      $\bigl(\lambda_i(G_n)\bigr)$ against the corresponding rearranged sampled values of $\widetilde G$.
      A superposition of the two plots as $n$ grows
      indicates, experimentally, asymptotic convergence of the empirical spectral distribution
      of $\{G_n\}_n$ to that of the symbol.

\item[\textbf{4.}] \textbf{Support analysis.}
      To estimate the measure of the zero set
      $\{(x,\theta):\widetilde G=0\}$ we count the fraction of eigenvalues below the threshold $0.1$ compared to the matrix size; the value is an estimation the complement of the predicted measure of the support
      as $n$ grows.
\end{itemize}

\medskip
\begin{remark}[Measure‑theoretic meaning of the rearrangement test]
\label{rearrangement}
As anticipated implicitly, the “sorted–eigenvalue versus sorted–symbol’’ comparison has a precise
measure‑theoretic foundation based on quantile approximation theory \cite{Bogoya2016}. For completeness, we report the construction in \cite{Bogoya2016} for the
unilevel setting $d=1$; the multilevel case is analogous.

Let $r$ denote the size of the matrix
$\widetilde G(\kappa,\xi)(x,\theta)$ and put
$D=[0,1]\times[-\pi,\pi]$ (so $\mu(D)=2\pi$).
Define the probability space defined by the triple
\[
  (\Omega,\mathcal F,\mathbb P),\qquad
  \Omega=D\times\{1,\dots,m\},\qquad
  \mathcal F=\mathcal B(D)\otimes2^{\{1,\dots,m\}},
\]
Here, $\mathcal B(D)$ denotes the Borel $\sigma$‑algebra on $D$, and
$2^{\{1,\dots,m\}}$ denotes the power set of $\{1,\dots,m\}$,
with product measure
\(
  \mathbb P(A\times B)=\dfrac{\mu(A)}{2\pi}\,\dfrac{|B|}{{r}}.
\)\\
Introduce the random variable
\[
  X(x,\theta,i):=\lambda_i\!\bigl(\widetilde G(\kappa,\xi)(x,\theta)\bigr),
  \qquad (x,\theta,i)\in\Omega,
\]
where the eigenvalues are ordered non‑decreasingly.
The quantile function of $X$ is precisely the \emph{non‑decreasing
rearrangement} $\widetilde G^{\dagger}$ of the matrix-valued symbol.\\
Because spectral distributions are defined only up to
measure-preserving rearrangements,
\[
  \{G_n\}_n\sim_{\lambda} \widetilde G
  \;\Longrightarrow\;
  \{G_n\}_n\sim_{\lambda} \widetilde G^{\dagger},
\]
so verifying convergence to the quantile $\widetilde G^{\dagger}$ is equivalent,
and numerically simpler, than verifying convergence to the symbol
$\widetilde G(\kappa,\xi)$ itself. \\
Further, whenever $\widetilde G^{\dagger}$ is continuous at $t\in(0,1)$, the sorted
eigenvalues satisfy
\(
   \lambda_{\lceil t d_n\rceil}(G_n)\to \widetilde G^{\dagger}(t)
\)
as $n\to\infty$, and the empirical proportion of eigenvalues below a small threshold
$t_0$ (we take $t_0=0.1$) converges to
\(
  \mathbb P[X\le t_0],
\)
therefore it can be used as an estimate of the measure of the rank‑deficient subset
$\{(x,\theta):\widetilde G(\kappa,\xi)(x,\theta)=0\}$.
\end{remark}

%====================================================================
\subsubsection{Numerical results}
\label{sec:numerical-results}
%====================================================================

In this section, we compare the sorted eigenvalues of each geometric mean matrix
\(G(A_n,B_n)\) with the rearranged eigenvalue distribution predicted by the
candidate symbol \(\widetilde{G}(\kappa,\xi)\). Plots are provided for each block size.

%-------------------------------------------------------------
\paragraph{Case 1.}
Figures~\ref{fig:case1_ex1}--\ref{fig:case1_ex2} display the ordered
eigenvalues of \(G(A_n,B_n)\) (colored markers) together with the ordered
samples of \(\widetilde{G}(\kappa,\xi)\) (solid red line).

\smallskip
\emph{Example 1.}
Because the supports of \(\kappa\) and \(\xi\) are essentially disjoint,
\(\widetilde{G}(\kappa,\xi)\equiv 0\).
Most eigenvalues match the zero line, showing convergence to the zero distribution. Only a small number of positive outliers remain above the line. For small values of $n$, the eigenvalue plot appears to be governed by a GLT {momentary} symbol that, when rearranged, follows a power law that rapidly collapses to the zero function as \(n \to \infty\).

\smallskip
\emph{Example 2.}
Here, the supports of \(\kappa\) and \(\xi\) intersect on \(\theta \in [-0.25,0.25]\). As a result, \(\widetilde{G}(\kappa,\xi)\) takes two values: zero outside the overlap and a positive constant matrix inside it. The eigenvalues reflect exactly this behavior: they form a long plateau at zero, followed by two shorter clusters at the positive eigenvalues determined by \(\widetilde{G}(\kappa,\xi)\). As we will see, the length of this zero plateau matches the respective Lebesgue measure of the complement of the support of \(\widetilde G\).

\begin{figure}[ht]
  \centering
  % First row
  \begin{subfigure}[t]{0.49\textwidth}
    \centering
    \includegraphics[width=\textwidth]{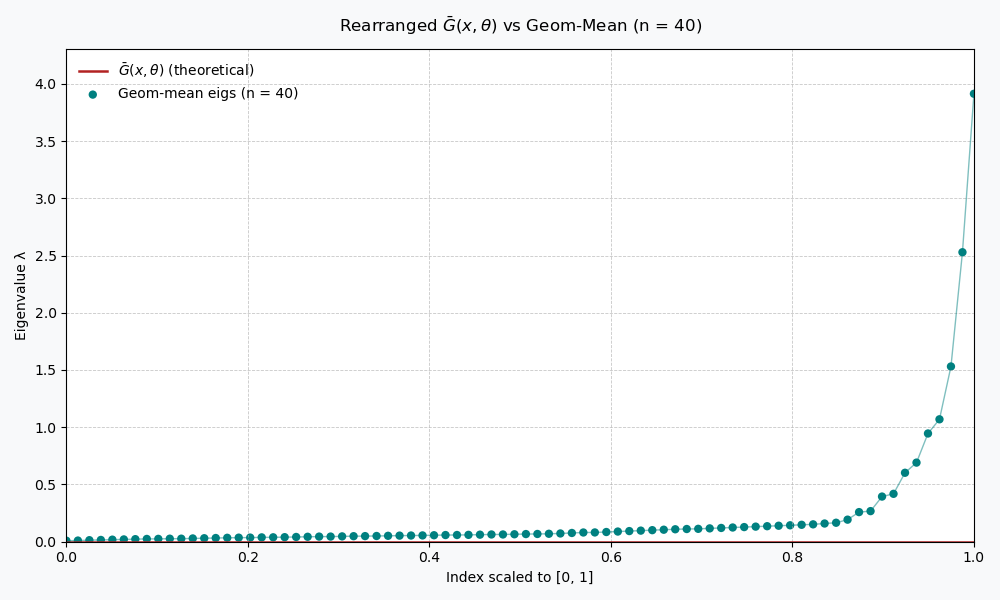}
  \end{subfigure}
  \hfill
  \begin{subfigure}[t]{0.49\textwidth}
    \centering
    \includegraphics[width=\textwidth]{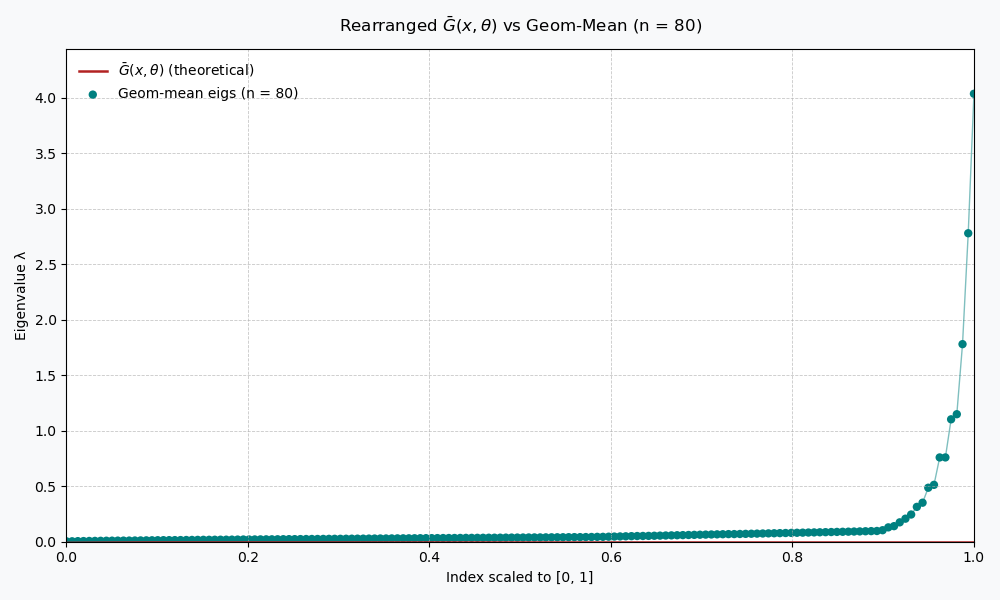}
  \end{subfigure}
  \\[-1em]  % Tighten vertical space

  % Second row
  \begin{subfigure}[t]{0.49\textwidth}
    \centering
    \includegraphics[width=\textwidth]{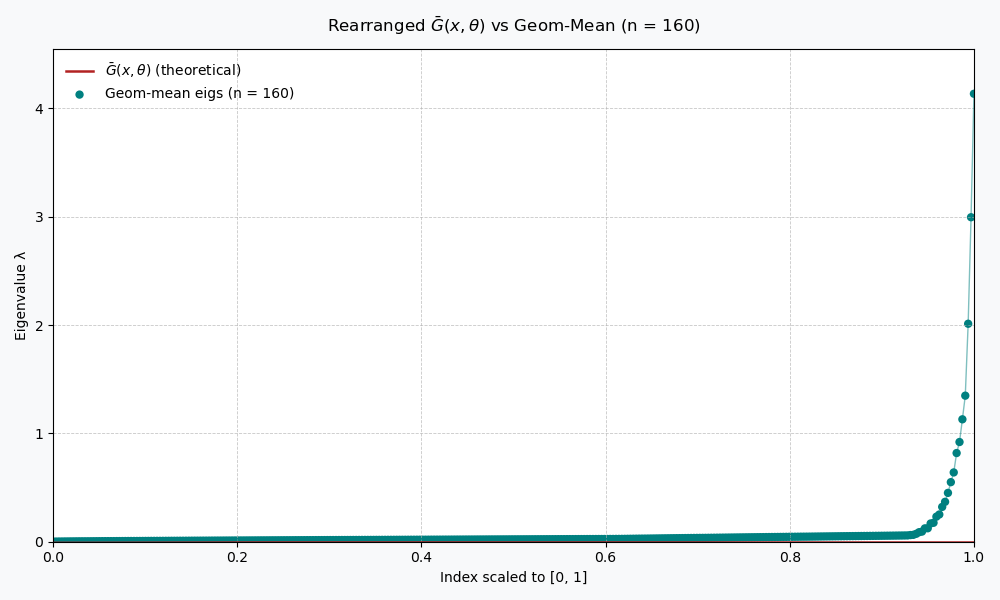}
  \end{subfigure}
  \hfill
  \begin{subfigure}[t]{0.49\textwidth}
    \centering
    \includegraphics[width=\textwidth]{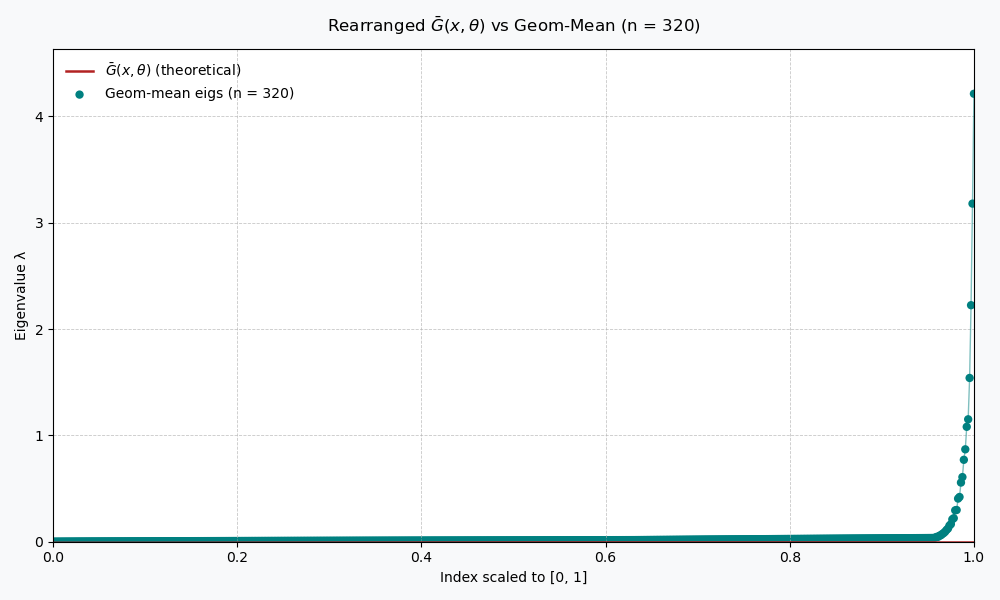}
  \end{subfigure}
\caption{\small{\textbf{Case 1 — Example 1.}
Sorted eigenvalues of \(G(A_n,B_n)\) (colored markers; \(n=40,80,160,320\))
versus the rearranged distribution of the symbol
\(\widetilde G(\kappa,\xi)\) (solid red line)}}.
\label{fig:case1_ex1}
\end{figure}

\begin{figure}[ht]
  \centering
  % First row
  \begin{subfigure}[t]{0.49\textwidth}
    \centering
    \includegraphics[width=\textwidth]{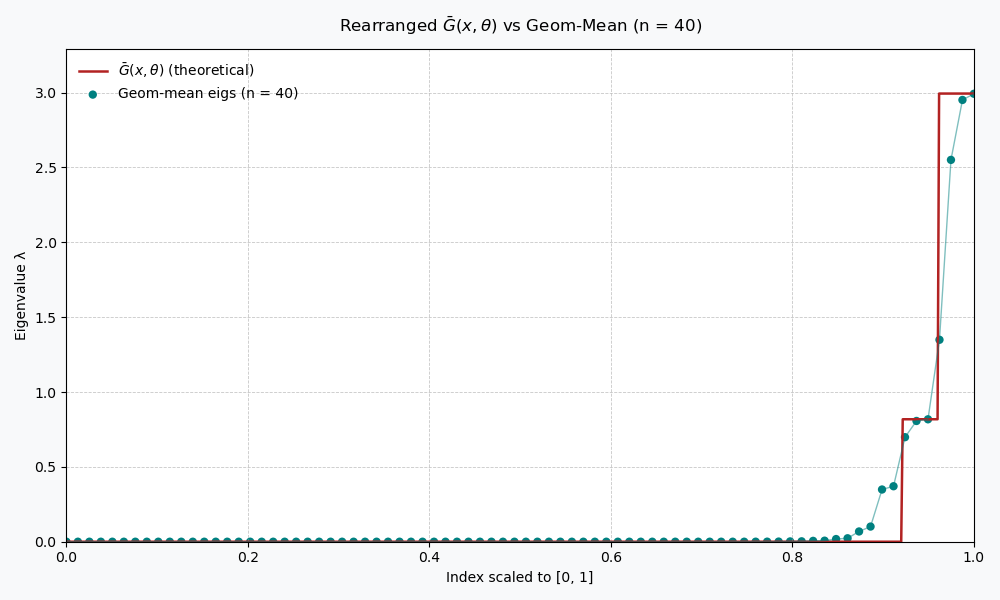}
  \end{subfigure}
  \hfill
  \begin{subfigure}[t]{0.49\textwidth}
    \centering
    \includegraphics[width=\textwidth]{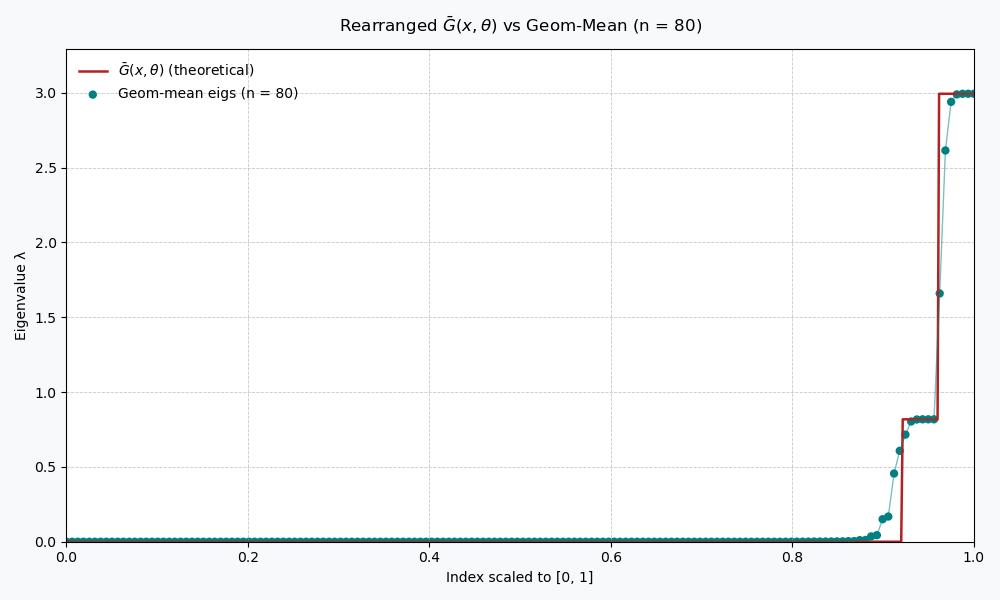}
  \end{subfigure}

  % Second row
  \begin{subfigure}[t]{0.49\textwidth}
    \centering
    \includegraphics[width=\textwidth]{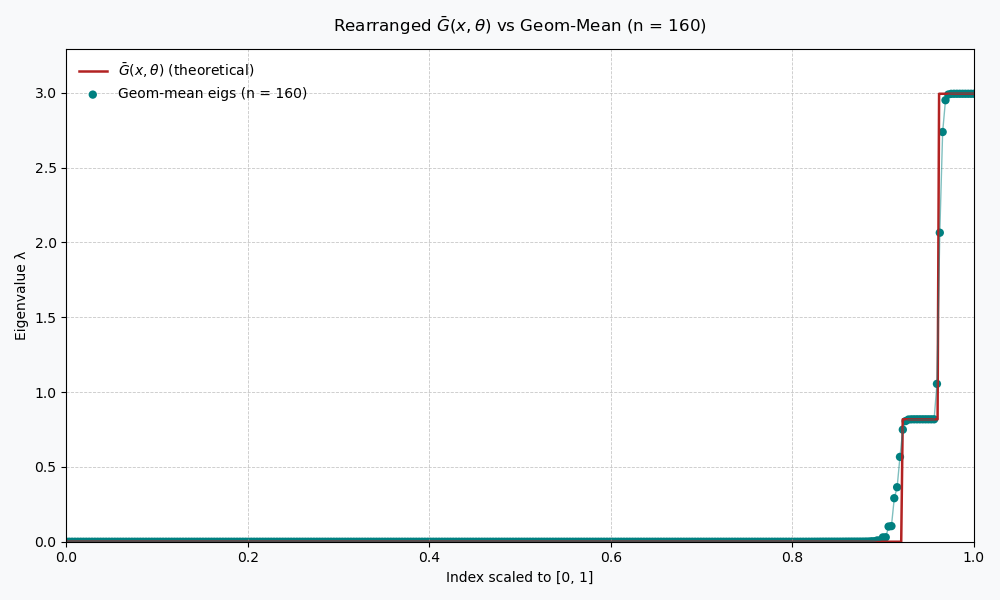}
  \end{subfigure}
  \hfill
  \begin{subfigure}[t]{0.49\textwidth}
    \centering
    \includegraphics[width=\textwidth]{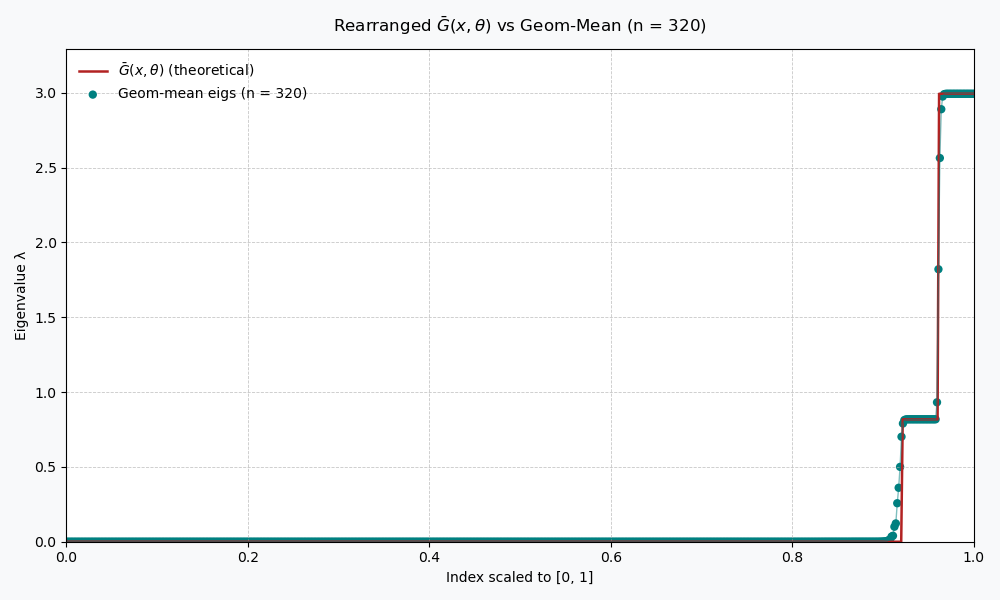}
  \end{subfigure}

  \caption{\small{\textbf{Case 1 — Example 2.}
Comparison of the ordered eigenvalues of \(G(A_n,B_n)\)
(colored markers) with the rearranged values of
\(\widetilde G(\kappa,\xi)\) (solid red line) for
\(n=40,80,160,320\).}}
  \label{fig:case1_ex2}
\end{figure}
\FloatBarrier

\paragraph{Case 2.}
Figures~\ref{fig:case2_ex1}--\ref{fig:case2_ex2} repeat the experiment for
symbols that are rank-deficient on sets of full measure.

\smallskip
\emph{Example 1.}
Because the supports are disjoint, we have \(\widetilde{G}(\kappa,\xi)\equiv 0\).
The eigenvalues converge to the zero symbol. Despite the momentary symbol may seem to cause a rougher perturbation, we remark that, unlike in Case 1 - Example 1, this perturbation does not result in proper outliers since all perturbed eigenvalues converge to zero. In practice, we observe a better convergence.

\smallskip
\emph{Example 2.}
Here, the intersection set has rank one, leading to exactly one positive eigenvalue
in \(\widetilde{G}(\kappa,\xi)\).
The plot shows convergence despite the rank deficiency of the support. Because the norm of the diagonal perturbation introduced in \ref{eq:case2_ex2} decays slowly, a momentary perturbation is still observed even in this non-zero case.

\begin{figure}[ht]
\captionsetup{font=small,skip=4pt}
  \centering
  % First row
  \begin{subfigure}[t]{0.49\textwidth}
    \centering
    \includegraphics[width=\textwidth]{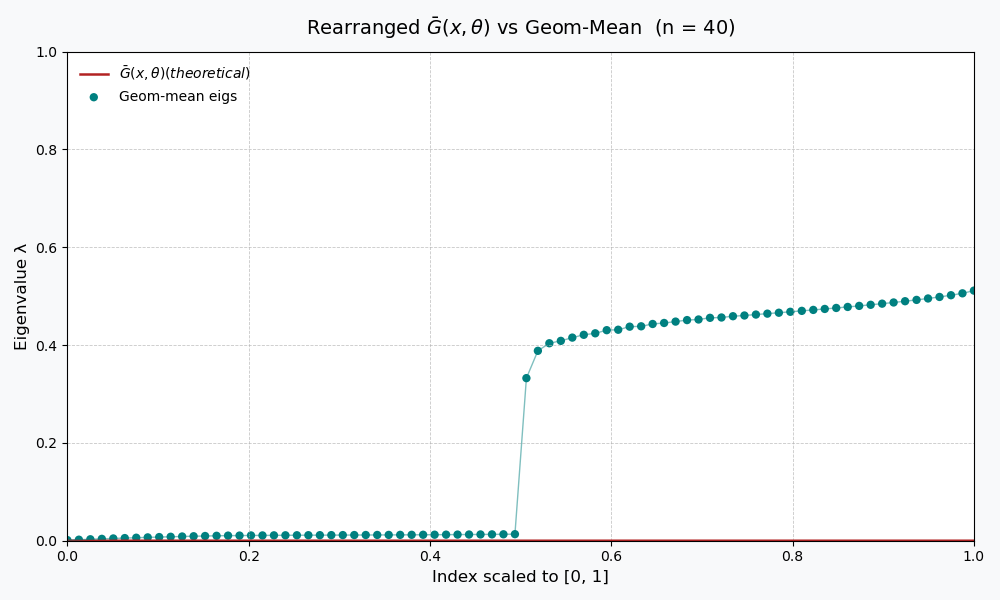}
  \end{subfigure}
  \hfill
  \begin{subfigure}[t]{0.49\textwidth}
    \centering
    \includegraphics[width=\textwidth]{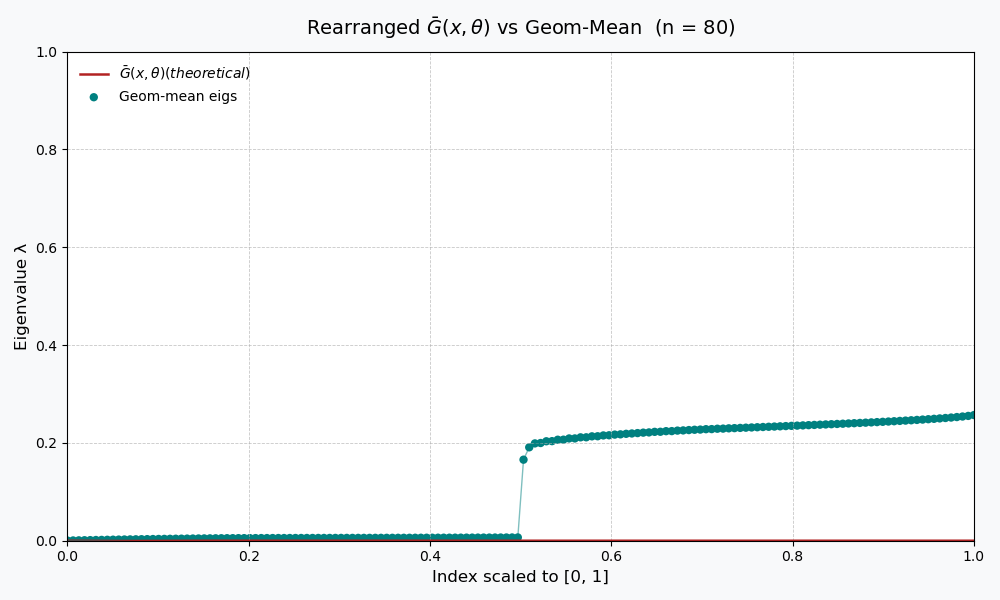}
  \end{subfigure}
  \\[-1em]  

  % Second row
  \begin{subfigure}[t]{0.49\textwidth}
    \centering
    \includegraphics[width=\textwidth]{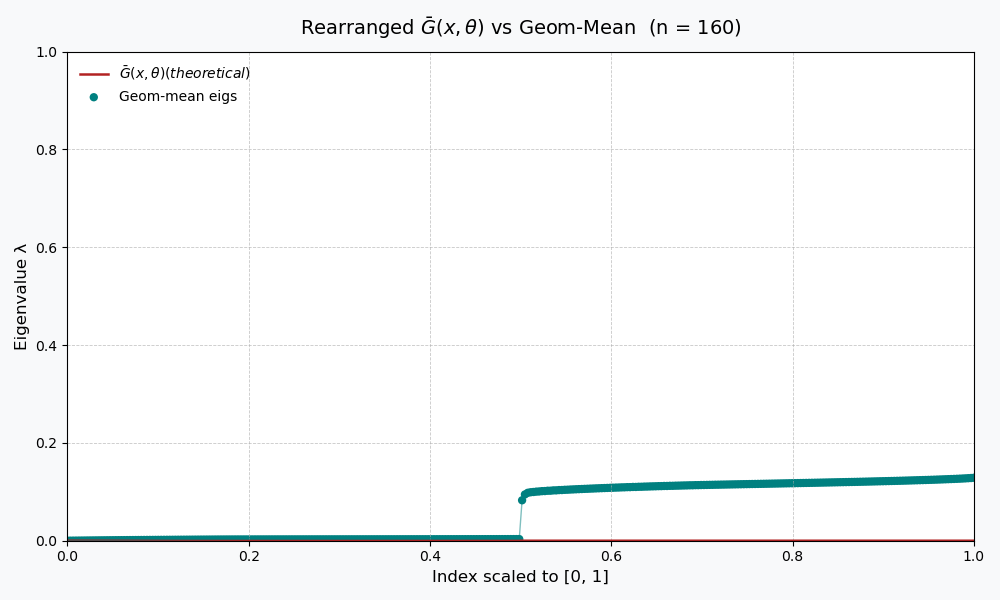}
  \end{subfigure}
  \hfill
  \begin{subfigure}[t]{0.49\textwidth}
    \centering
    \includegraphics[width=\textwidth]{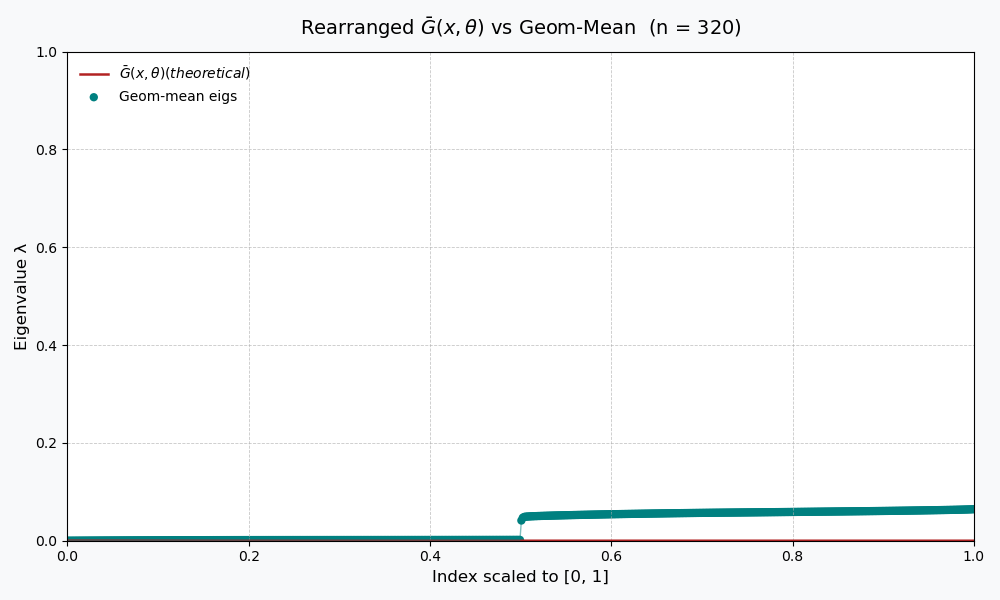}
  \end{subfigure}

  \caption{\small{\textbf{Case 2, Example 1.} Sorted eigenvalues of \(G(A_n,B_n)\) (colored markers) versus the rearranged eigenvalue distribution predicted by \(G(\kappa,\xi)\) (solid red line), for
\(n=40,80,160,320\).}}
  \label{fig:case2_ex1}
\end{figure}

\begin{figure}[ht]
  \centering
  % First row
  \begin{subfigure}[t]{0.49\textwidth}
    \centering
    \includegraphics[width=\textwidth]{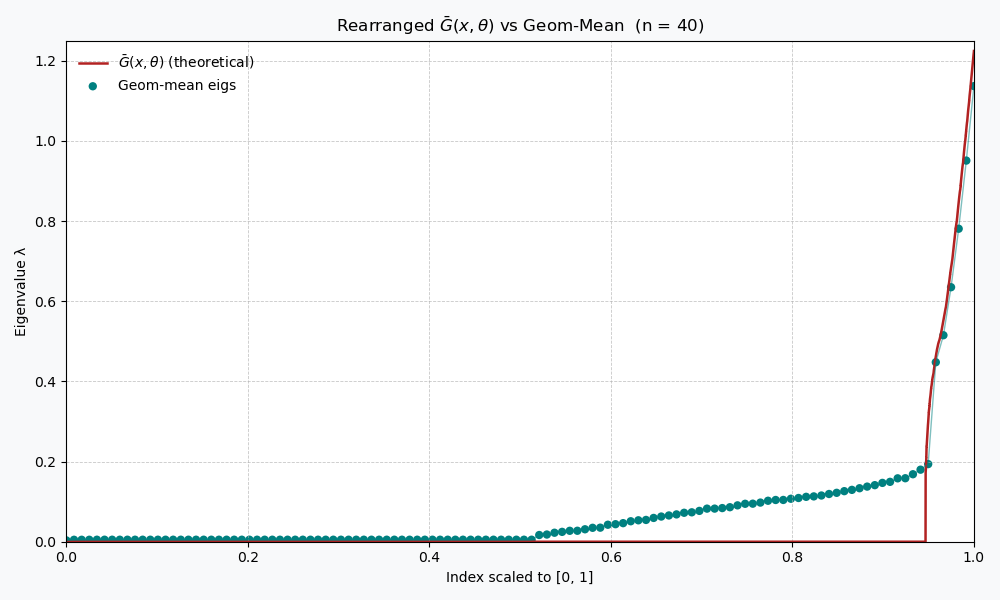}
  \end{subfigure}
  \hfill
  \begin{subfigure}[t]{0.49\textwidth}
    \centering
    \includegraphics[width=\textwidth]{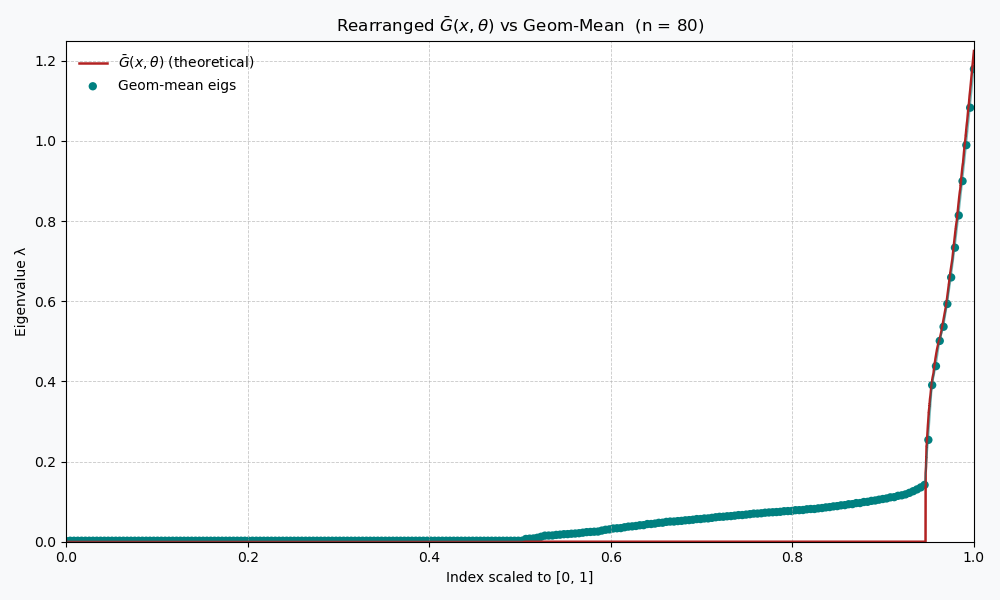}
  \end{subfigure}
  \\[-1em]

  % Second row
  \begin{subfigure}[t]{0.49\textwidth}
    \centering
    \includegraphics[width=\textwidth]{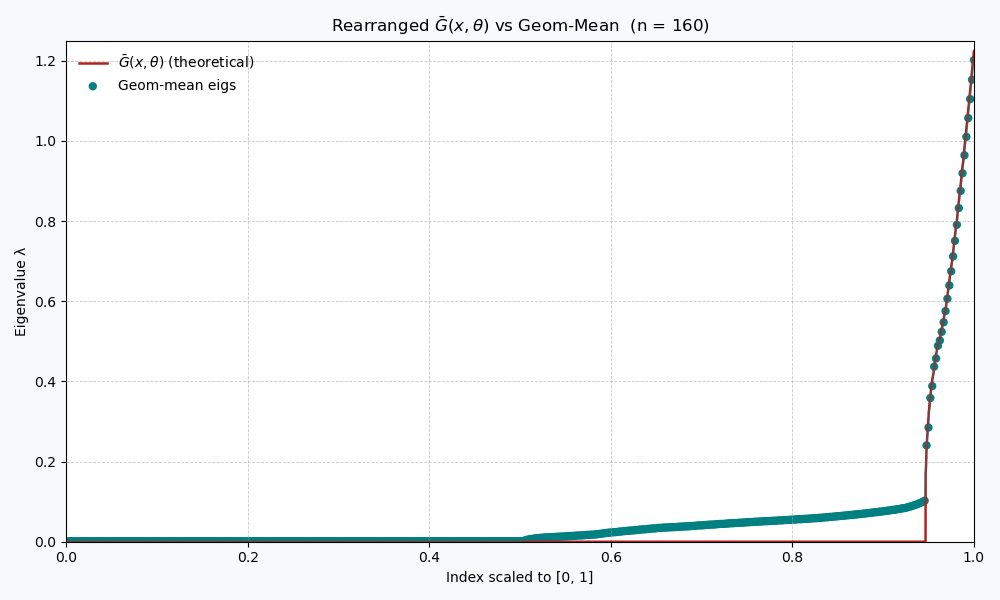}
  \end{subfigure}
  \hfill
  \begin{subfigure}[t]{0.49\textwidth}
    \centering
    \includegraphics[width=\textwidth]{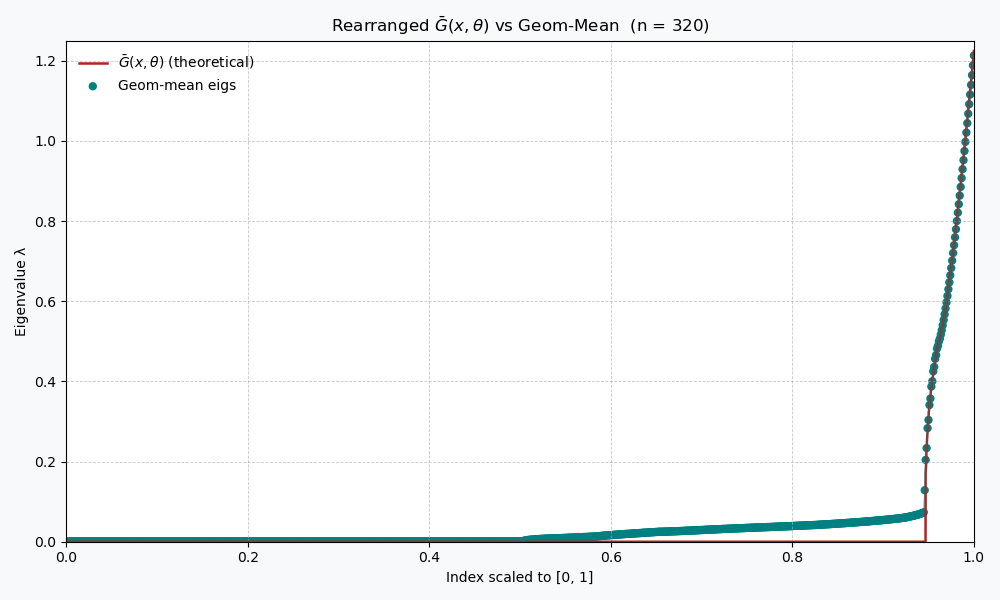}
  \end{subfigure}

  \caption{\small{\textbf{Case 2, Example 2.} Sorted eigenvalues of \(G(A_n,B_n)\) (colored markers) versus the rearranged eigenvalue distribution from \(G(\kappa,\xi)\) (solid red line), for
\(n=40,80,160,320\).}}
  \label{fig:case2_ex2}
\end{figure}
\FloatBarrier 
\paragraph{Extremal eigenvalues.}
For each block size \(n\) we report
\[
  \lambda_{\min}(G_n), \qquad
  \lambda_{\max}(G_n), \qquad
  \mathrm{cond}_2(G_n)=\frac{\lambda_{\max}(G_n)}{\lambda_{\min}(G_n)},
\]
together with the essential infimum and supremum of the conjectured
symbol, \(\widetilde G_{\min}\) and \(\widetilde G_{\max}\).
In this way, we can observe how rapidly the minimum eigenvalue decays and, in most cases, how the maximum eigenvalues converge to the essential supremum of the spectral symbol.
This provides a clear measure of the conditioning behaviour as \(n\) grows.

% ------------------------------------------------------------------
% --------------------  CASE 1  ------------------------------------
% ------------------------------------------------------------------
\medskip\noindent

\begin{table}[H]
  \centering
  \caption{Extremal eigenvalues and condition numbers for
           \textbf{Case 1 — Example 1} (symbol zero almost everywhere).}
  \label{tab:case1_ex1_range}
  \begin{tabular}{|c
                  |S[table-format=1.4]
                  |S[table-format=1.4]
                  |S[table-format=1.4e+2]
                  |S[table-format=1.8]
                  |S[table-format=1.4e+2]|}
    \hline
    {$n$} & {\(\widetilde G_{\min}\)}
           & {\(\widetilde G_{\max}\)}
           & {Min.\ eig.}
           & {Max.\ eig.}
           & {\(\mathrm{cond}_2(G_n)\)}\\
    \hline
    40  & 0.0000 & 0.0000 & 8.7142e-03 & 3.91278029 & 4.490e+02 \\
    80  & 0.0000 & 0.0000 & 3.6849e-03 & 4.03656685 & 1.095e+03 \\
    160 & 0.0000 & 0.0000 & 1.4929e-03 & 4.13399129 & 2.769e+03 \\
    320 & 0.0000 & 0.0000 & 5.8335e-04 & 4.21189805 & 7.220e+03 \\
    \hline
  \end{tabular}
\end{table}

\paragraph{Case 1 — Example 2.}
Here \(\widetilde G_{\max}=2.9939\) is strictly positive. In this example we observe a fast convergence  of \(\lambda_{\max}(G_n)\), and an exponential decaying of
\(\lambda_{\min}(G_n)\), and this clearly causes the escalation of
\(\mathrm{cond}_2(G_n)\) reaching
\(\mathcal{O}(10^{8})\) at \(n=320\).

\begin{table}[H]
  \centering
  \caption{Extremal eigenvalues and condition numbers for
           \textbf{Case 1 — Example 2} (symbol positive on an overlap set).}
  \label{tab:case1_ex2_range}
  \begin{tabular}{|c
                  |S[table-format=1.4]
                  |S[table-format=1.8]
                  |S[table-format=1.4e+2]
                  |S[table-format=1.8]
                  |S[table-format=1.4e+2]|}
    \hline
     {$n$} & {\(\widetilde G_{\min}\)}
            & {\(\widetilde G_{\max}\)}
            & {Min.\ eig.}
            & {Max.\ eig.}
            & {\(\mathrm{cond}_2(G_n)\)}\\
    \hline
    40  & 0.0000 & 2.99393066 & 1.5625e-05 & 2.99257415 & 1.915e+05 \\
    80  & 0.0000 & 2.99393066 & 1.9531e-06 & 2.99393280 & 1.533e+06 \\
    160 & 0.0000 & 2.99393066 & 2.4414e-07 & 2.99393094 & 1.226e+07 \\
    320 & 0.0000 & 2.99393066 & 3.0518e-08 & 2.99393070 & 9.811e+07 \\
    \hline
  \end{tabular}
\end{table}

% ------------------------------------------------------------------
% Case 2 – Example 1  (symbol zero almost everywhere)
% ------------------------------------------------------------------

\medskip\noindent
\smallskip
\paragraph{Case 2 — Example 1.}
The minimum eigenvalue seems to decay to 0 at a quadratic rate, while the linear convergence rate to 0 of the largest eigenvalue is imposed by the momentary symbol. As a consequence, the growth of
\(\mathrm{cond}_2(G_n)\) is mild.
\medskip
\smallskip

\begin{table}[H]
  \centering
  \caption{Extremal eigenvalues and condition numbers for
           \textbf{Case 2 — Example 1} (candidate symbol
           \(\widetilde G\equiv0\)).}
  \label{tab:case2_ex1_range}
  \begin{tabular}{|c
                  |S[table-format=1.2]
                  |S[table-format=1.2]
                  |S[table-format=1.2e+2]
                  |S[table-format=1.8] 
                  |S[table-format=1.2e+2]|}
    \hline
     {$n$} & {\(\widetilde G_{\min}\)}
            & {\(\widetilde G_{\max}\)}
            & {Min.\ eig.}
            & {Max.\ eig.}
            & {\(\mathrm{cond}_2(G_n)\)}\\
    \hline
    40  & 0.00 & 0.00 & 1.2615e-03 & 0.51134401 & 4.0536e+02 \\
    80  & 0.00 & 0.00 & 3.2103e-04 & 0.25691974 & 8.0029e+02 \\
    160 & 0.00 & 0.00 & 8.0993e-05 & 0.12877171 & 1.5899e+03 \\
    320 & 0.00 & 0.00 & 2.0342e-05 & 0.06447364 & 3.1695e+03 \\
    \hline
  \end{tabular}
\end{table}
\medskip
\smallskip

\paragraph{Case 2 — Example 2.}
Here \(\widetilde G_{\max}\simeq1.22\).
\(\lambda_{\max}(G_n)\) converges to this value and is included in the range of $\widetilde G$, whereas
\(\lambda_{\min}(G_n)\) decays linearly. This better conditioning compared to Example 2 of the first case is reasonable, since the positive perturbations we introduced to the Toeplitz matrices of this example (Equation \ref{eq:case2_ex2} decay in norm more slowly than the ones in Case 1 - Example 1. 

\medskip

% ------------------------------------------------------------------
% Case 2 – Example 2  (symbol positive on overlap)
% ------------------------------------------------------------------
\begin{table}[H]
  \centering
  \caption{Extremal eigenvalues and condition numbers for
           \textbf{Case 2 — Example 2} (symbol positive on an overlap set).}
  \label{tab:case2_ex2_range}
  \begin{tabular}{|c
                  |S[table-format=1.2]
                  |S[table-format=1.8]
                  |S[table-format=1.2e+2]
                  |S[table-format=1.8]
                  |S[table-format=1.2e+2]|}
    \hline
     {$n$} & {\(\widetilde G_{\min}\)}
            & {\(\widetilde G_{\max}\)}
            & {Min.\ eig.}
            & {Max.\ eig.}
            & {\(\mathrm{cond}_2(G_n)\)}\\
    \hline
    40  & 0.00 & 1.22418242 & 3.24e-03 & 1.13658304 & 3.50e+02 \\
    80  & 0.00 & 1.22418242 & 1.46e-03 & 1.17911634 & 8.10e+02 \\
    160 & 0.00 & 1.22418242 & 5.65e-04 & 1.20156562 & 2.13e+03 \\
    320 & 0.00 & 1.22418242 & 8.19e-05 & 1.21306699 & 1.48e+04 \\
    \hline
  \end{tabular}
\end{table}

%--------------------------------------------------------------------
\paragraph{Zero‑related statistics.}
For each block size \(n\) we report the fraction of eigenvalues whose
modulus does not exceed the fixed cutoff \(0.1\); see
Tables~\ref{tab:case1_ex1_zero}–\ref{tab:case2_ex2_zero}.
Each table lists the empirical proportion
(\emph{Prop.\(\le0.1\)}), the theoretical measure of the zero set of the
rearranged symbol (\emph{Target}), and the absolute difference (\emph{Error}).

% ------------------------------------------------------------------
% --------------------  CASE 1  ------------------------------------
% ------------------------------------------------------------------
\smallskip
\paragraph{Case 1 — Example 1.}
The conjectured symbol is the zero function, so the target measure is~1.\\
The empirical value increases monotonically with~\(n\), yet within a small error.
Because the number of outliers is at most \(o(n)\), we expect that the proportion will converge to 1 as $n \to \infty $.

\paragraph{Case 1 — Example 2.}
Here, the zero set of the rearranged spectral symbol has measure
\[
  1-\frac{1}{4\pi}\approx 0.9204.
\]%
The measured proportions appear to approach this value with an
\(\mathcal O(n^{-1})\) decay, the error falling below \(10^{-3}\) by
\(n=160\).

% ------------------------------------------------------------------
% --------------------  CASE 2  ------------------------------------
% ------------------------------------------------------------------
\medskip
\paragraph{Case 2 — Example 1.}
Here the expected spectral symbol is zero, but the momentary symbol perturbation,
as observed in plot~\ref{fig:case2_ex1}, produces an apparent
plateau \(0.5\) up to \(n=320\), where the proportion jumps to~1 and the measured error vanishes.
Since the momentary perturbation does not create proper outliers, as observed in Case 1 Ex.\,2, a faster convergence can reasonably be expected in this case.

\paragraph{Case 2 — Example 2.}
In this case, the complement of the support of the rearranged distribution has measure
\[
  \frac{17}{18}\approx 0.9444.
\]
The empirical measure of the zero cluster converges to this value with an error that appears to follow the same power law observed in Case 2 Ex.\,1.

\medskip
\begin{table}[H]
  \centering
  \caption{Case 1, Example 1: proportion of eigenvalues below \(0.1\); target \(0\).}
  \label{tab:case1_ex1_zero}
  \begin{tabular}{|c|S|S|S|}
    \hline
    {$n$} & {Prop.\(\le0.1\)} & {Target} & {Error} \\
    \hline
    40  & 0.6375 & 1.0000 & 0.3625 \\
    80  & 0.8938 & 1.0000 & 0.1062 \\
    160 & 0.9438 & 1.0000 & 0.0562 \\
    320 & 0.9688 & 1.0000 & 0.0312 \\
    \hline
  \end{tabular}
\end{table}

\begin{table}[H]
  \centering
  \caption{Case 1, Example 2: proportion of eigenvalues below \(0.1\);
           target \(1-\frac{1}{4\pi}\approx0.92042\).}
  \label{tab:case1_ex2_zero}
  \begin{tabular}{|c|S|S|S|}
    \hline
    {$n$} & {Prop.\(\le0.1\)} & {Target} & {Error} \\
    \hline
    40  & 0.8750 & 0.9204 & 0.0454 \\
    80  & 0.8938 & 0.9204 & 0.0266 \\
    160 & 0.9031 & 0.9204 & 0.0173 \\
    320 & 0.9109 & 0.9204 & 0.0095 \\
    \hline
  \end{tabular}
\end{table}

\begin{table}[H]
  \centering
  \caption{Case 2, Example 1: proportion of eigenvalues below \(0.1\); target \(0\).}
  \label{tab:case2_ex1_zero}
  \begin{tabular}{|c|S|S|S|}
    \hline
    {$n$} & {Prop.\(\le0.1\)} & {Target} & {Error} \\
    \hline
    40  & 0.5000 & 1.0000 & 0.5000 \\
    80  & 0.5000 & 1.0000 & 0.5000 \\
    160 & 0.5156 & 1.0000 & 0.4844 \\
    320 & 1.0000 & 1.0000 & 0.0000 \\
    \hline
  \end{tabular}
\end{table}

\begin{table}[H]
  \centering
  \caption{Case 2, Example 2: proportion of eigenvalues below \(0.1\);
           target \(\tfrac{17}{18}\approx0.9444\).}
  \label{tab:case2_ex2_zero}
  \begin{tabular}{|c|S|S|S|}
    \hline
    {$n$} & {Prop.\(\le0.1\)} & {Target} & {Error} \\
    \hline
    40  & 0.7667 & 0.9444 & 0.1777 \\
    80  & 0.8833 & 0.9444 & 0.0611 \\
    160 & 0.9438 & 0.9444 & 0.0006 \\
    320 & 0.9448 & 0.9444 & 0.0004 \\
    \hline
  \end{tabular}
\end{table}

\section{Conclusion}\label{sec concl}

We have studied the spectral distribution of the geometric mean matrix-sequence of two $d$-level $r$-block GLT matrix-sequences 
$\{G(A_n, B_n)\}_n$  formed by Hermitian Positive Definite (HPD) matrices, where $\{A_n\}_n, \{B_n)\}_n$ have GLT symbols $\kappa, \xi$, respectively.
In Theorem \ref{theorem 1} and in Theorem \ref{theorem 2} we have shown that the assumption that at least one of the input GLT symbols is invertible a.e. is not necessary, when the symbols commute so that 
\[
\{G(A_n, B_n)\}_n \sim_{\mathrm{GLT}} G(\kappa,\xi)\equiv (\kappa \xi)^{1/2}.
\]
On the other hand, the statement is generally false or even not well posed when the symbols are not invertible a.e. and do not commute, as shown in detail in several numerical experiments. Further numerical experiments are presented and critically commented in connection with extremal spectral features, linear positive operators, and in connection with the notion of Toeplitz and GLT momentary symbols. What the numerics show opens the door to further theoretical studies, which we will consider in the near future.


\begin{thebibliography}{99}

\bibitem{ahmad2025matrix}
Ahmad, D., Khan, M.F., Serra-Capizzano, S. 
Matrix-Sequences of Geometric Means in the Case of Hidden (Asymptotic) Structures. 
\textit{Mathematics} \textbf{2025}, \emph{13}, 393. 
\url{https://doi.org/10.3390/math13030393}.

\bibitem{ando2004geometric} Ando, T., Li, C.-K., Mathias, R.: Geometric means. \textit{Linear Algebra Appl.} 385, 305–334 (2004).

\bibitem{new momentary} 
Barakitis, N., Loi, V., Serra-Capizzano, S. 
A note on eigenvalues and singular values of variable Toeplitz matrices and matrix-sequences, with application to variable two-step BDF approximations to parabolic equations. 
\textbf{2025}, \textit{In press},
\url{https://www.brownsbfs.co.uk/Product/Asharaf-Noufal/Recent-Developments-in-Spectral-and-Approximation-Theory/9783031902390}



\bibitem{Barbarino2020a}
Barbarino, G., Garoni, C., Serra-Capizzano, S. 
Block generalized locally Toeplitz sequences: Theory and applications in the unidimensional case. 
\textit{Electron. Trans. Numer. Anal.} \textbf{2020}, \emph{53}, 28–112.

\bibitem{Barbarino2020b}
Barbarino, G., Garoni, C., Serra-Capizzano, S. 
Block generalized locally Toeplitz sequences: Theory and applications in the multidimensional case. 
\textit{Electron. Trans. Numer. Anal.} \textbf{2020}, \emph{53}, 113–216.

\bibitem{Barbarino2022}
Barbarino, G. 
A systematic approach to reduced GLT. 
\textit{BIT} \textbf{2022}, \emph{62}, 681–743.

\bibitem{batchelor2005rigorous}
Batchelor, P.G., Moakher, M., Atkinson, D., Calamante, F., Connelly, A. 
A rigorous framework for diffusion tensor calculus. 
\textit{Magn. Reson. Med.} \textbf{2005}, \emph{53}(1), 221–225.

\bibitem{extr1-glt}
Beckermann, B., Serra-Capizzano, S. 
On the asymptotic spectrum of finite element matrix sequences. 
\textit{SIAM J. Numer. Anal.} \textbf{2007}, \emph{45}(2), 746–769.


\bibitem{Bhatia1997}
Bhatia, R. 
\textit{Matrix Analysis}; Graduate Texts in Mathematics, Vol. 169; Springer-Verlag: New York, 1997.

\bibitem{Bhatia2007} 
Bhatia, R. 
\textit{Positive Definite Matrices}; Princeton Series in Applied Mathematics; Princeton University Press: Princeton, NJ, USA, 2007.

\bibitem{Bini2024} 
Bini, D.A., Iannazzo, B. 
Computational aspects of the geometric mean of two matrices: A survey. 
\textit{Acta Sci. Math.} \textbf{2024}, \emph{90}, 349–389. 
\url{https://doi.org/10.1007/s44146-024-00155-5}.

\bibitem{bini2013computing}
Bini, D.A., Iannazzo, B. 
Computing the Karcher mean of symmetric positive definite matrices. 
\textit{Linear Algebra Appl.} \textbf{2013}, \emph{438}(4), 1700–1710.

\bibitem{Bogoya2016} Bogoya, J.M., Böttcher, A., Maximenko, E.A. From convergence in distribution to uniform convergence. \textit{Bol. Soc. Mat. Mex.} \textbf{2016}, \emph{22}, 695–710.



\bibitem{bolten2022toeplitz}
Bolten, M., Ekström, S.-E., Furci, I., Serra-Capizzano, S. 
Toeplitz momentary symbols: Definition, results, and limitations in the spectral analysis of structured matrices.
\textit{Linear Algebra Appl.} \textbf{2022}, \emph{651}, 135–168.
\url{https://doi.org/10.1016/j.laa.2022.06.017}.

\bibitem{bolten2023note}
Bolten, M., Ekström, S.-E., Furci, I., Serra-Capizzano, S. 
A note on the spectral analysis of matrix sequences via GLT momentary symbols: From all-at-once solution of parabolic problems to distributed fractional order matrices.
\textit{Electron. Trans. Numer. Anal.} \textbf{2023}, \emph{58}, 136–163.
\url{https://doi.org/10.1553/etna_vol58s136}.

\bibitem{Bottcher1998}
Böttcher, A., Grudsky, S.M. 
On the condition numbers of large semi-definite Toeplitz matrices. 
\textit{Linear Algebra Appl.} \textbf{1998}, \emph{279}, 285–301.

\bibitem{fasi2018computing}
Fasi, M., Iannazzo, B. 
Computing the weighted geometric mean of two large-scale matrices and its inverse times a vector. 
\textit{SIAM J. Matrix Anal. Appl.} \textbf{2018}, \emph{39}(1), 178–203.

\bibitem{garoni2017}
Garoni, C., Serra-Capizzano, S. 
\textit{Generalized Locally Toeplitz Sequences: Theory and Applications. Vol. I}; 
Springer: Cham, Switzerland, 2017.

\bibitem{garoni2018}
Garoni, C., Serra-Capizzano, S. 
\textit{Generalized Locally Toeplitz Sequences: Theory and Applications. Vol. II}; 
Springer: Cham, Switzerland, 2018.

\bibitem{iannazzo2019derivative}
Iannazzo, B., Jeuris, B., Pompili, F. 
The derivative of the matrix geometric mean with an application to the nonnegative decomposition of tensor grids. 
In \textit{Structured Matrices in Numerical Linear Algebra}; Springer: Cham, Switzerland, 2019; pp. 107–128.

\bibitem{kubo1980means}
Kubo, F., Ando, T. 
Means of positive linear operators. 
\textit{Math. Ann.} \textbf{1979/1980}, \emph{246}(3), 205–224.

\bibitem{lapuyade2008radar}
Lapuyade-Lahorgue, J., Barbaresco, F. 
Radar detection using Siegel distance between autoregressive processes, application to HF and X-band radar. 
In \textit{Proceedings of the 2008 IEEE Radar Conference}; IEEE: Rome, Italy, 2008; pp. 1–6.

\bibitem{moakher2005differential}
Moakher, M. 
A differential geometric approach to the geometric mean of symmetric positive-definite matrices. 
\textit{SIAM J. Matrix Anal. Appl.} \textbf{2005}, \emph{26}(3), 735–747.

\bibitem{moakher2006averaging}
Moakher, M. 
On the averaging of symmetric positive-definite tensors. 
\textit{J. Elasticity} \textbf{2006}, \emph{82}(3), 273–296.

\bibitem{Nakamura2009}
Nakamura, N. 
Geometric means of positive operators. 
\textit{Kyungpook Math. J.} \textbf{2009}, \emph{49}(1), 167–181.

\bibitem{Noutsos2008}
Noutsos, D., Serra-Capizzano, S., Vassalos, P. 
The conditioning of FD matrix sequences coming from semi-elliptic differential equations. 
\textit{Linear Algebra Appl.} \textbf{2008}, \emph{428}, 600–624.

\bibitem{pusz1975functional}
Pusz, W., Woronowicz, S.L. 
Functional calculus for sesquilinear forms and the purification map. 
\textit{Rep. Math. Phys.} \textbf{1975}, \emph{8}(2), 159–170.

\bibitem{Rathi2007}
Rathi, Y., Tannenbaum, A., Michailovich, O. 
Segmenting images on the tensor manifold. 
In \textit{Proceedings of the 2007 IEEE Conference on Computer Vision and Pattern Recognition}; IEEE: San Francisco, CA, USA, 2007; pp. 1–8.


\bibitem{SerraCapizzano1996}
Serra-Capizzano, S. 
On the extreme spectral properties of Toeplitz matrices generated by $L_1$ functions with several minima/maxima. 
\textit{BIT} \textbf{1996}, \emph{36}, 135–142.

\bibitem{SerraCapizzano1998}
Serra-Capizzano, S. 
On the extreme eigenvalues of Hermitian (block) Toeplitz matrices. 
\textit{Linear Algebra Appl.} \textbf{1998}, \emph{270}, 109–129.

\bibitem{SerraCapizzano1999a}
Serra-Capizzano, S. 
Asymptotic results on the spectra of block Toeplitz preconditioned matrices. 
\textit{SIAM J. Matrix Anal. Appl.} \textbf{1999}, \emph{20}, 31–44.

\bibitem{SerraCapizzano1999b}
Serra-Capizzano, S. 
Spectral and computational analysis of block Toeplitz matrices having nonnegative definite matrix-valued generating functions. 
\textit{BIT} \textbf{1999}, \emph{39}, 152–175.

\bibitem{LPO-rev}
Serra-Capizzano, S. 
Some theorems on linear positive operators and functionals and their applications.
\textit{Comp. Math. Appl.} \textbf{2000}, \emph{39}(7-8), 139–167.

\bibitem{SerraCapizzano2001}
Serra-Capizzano, S. 
Spectral behavior of matrix sequences and discretized boundary value problems. 
\textit{Linear Algebra Appl.} \textbf{2001}, \emph{337}, 37–78.

\bibitem{serra2003generalized}
Serra-Capizzano, S. 
Generalized locally Toeplitz sequences: Spectral analysis and applications to discretized partial differential equations. 
\textit{Linear Algebra Appl.} \textbf{2003}, \emph{366}, 371–402.

\bibitem{SerraCapizzano2006}
Serra-Capizzano, S. 
The GLT class as a generalized Fourier analysis and applications. 
\textit{Linear Algebra Appl.} \textbf{2006}, \emph{419}, 180–233.

\bibitem{extr2-glt}
Serra-Capizzano, S., Tablino-Possio, C. 
Spectral and structural analysis of high precision finite difference matrices for elliptic operators. 
\textit{Linear Algebra Appl.} \textbf{1999}, \emph{293}(1-3), 85–131.


\bibitem{SerraCapizzanoTilli1999}
Serra-Capizzano, S., Tilli, P. 
Extreme singular values and eigenvalues of non-Hermitian block Toeplitz matrices. 
\textit{J. Comput. Appl. Math.} \textbf{1999}, \emph{108}, 113–130.

\bibitem{SerraCapizzanoTilli-LPO}
Serra-Capizzano, S., Tilli, P. 
On unitarily invariant norms of matrix-valued linear positive operators. 
\textit{J. Inequal. Appl.} \textbf{2002}, \emph{7}(3), 309–330.


\bibitem{tilli1998locally}
Tilli, P. 
Locally Toeplitz sequences: Spectral properties and applications. 
\textit{Linear Algebra Appl.} \textbf{1998}, \emph{278}, 91–120.

\bibitem{tyrtyshnikov1996unifying}
Tyrtyshnikov, E.E. 
A unifying approach to some old and new theorems on distribution and clustering. 
\textit{Linear Algebra Appl.} \textbf{1996}, \emph{232}, 1–43.

\bibitem{Vassalos2018}
Vassalos, P. 
Asymptotic results on the condition number of FD matrices approximating semi-elliptic PDEs. 
\textit{Electron. J. Linear Algebra} \textbf{2018}, \emph{34}, 566–581.

\bibitem{yang2010geometry}
Yang, L., Arnaudon, M., Barbaresco, F. 
Geometry of covariance matrices and computation of median. 
In \textit{Bayesian Inference and Maximum Entropy Methods in Science and Engineering}, 
AIP Conf. Proc. \textbf{1305}; American Institute of Physics: Melville, NY, USA, 2010; pp. 479–486.

\bibitem{yger2017riemannian}
Yger, F., Berar, M., Lotte, F. 
Riemannian approaches in brain-computer interfaces: A review. 
\textit{IEEE Trans. Neural Syst. Rehabil. Eng.} \textbf{2017}, \emph{25}(10), 1753–1762.

\end{thebibliography}
\end{document}